\documentclass[journal,12pt,draftcls,onecolumn]{IEEEtran}

\usepackage{graphicx}
\usepackage{amssymb}
\usepackage{mathrsfs}
\usepackage{cite}
\usepackage{amsmath}
\usepackage{amsthm}
\usepackage[table]{xcolor}
\usepackage{tabularx}

\usepackage{enumitem}
\usepackage{epstopdf}
\usepackage{caption}
\usepackage{subcaption}
\usepackage{slashbox}
\newcommand{\tran}{\mathrm{T}}
\newcommand{\bsm}{\begin{smallmatrix}}
\newcommand{\esm}{\end{smallmatrix}}
\newcommand{\bbm}{\begin{bmatrix}}
\newcommand{\ebm}{\end{bmatrix}}
\newcommand{\rank}{\mathrm{rank}}

\newcommand{\ssp}[1]{\mathscr{#1}}      
\newcommand{\fld}[1]{\mathbb{#1}}       
\newcommand{\op}[1]{\mathcal{#1}}       
\newcommand{\spanset}[1]{\ \mathrm{span}\{#1\}} 
\newcommand{\pardiff}[2]{\frac{\partial {#1}}{\partial {#2}}}

\newcommand{\RieszBas}[1]{\{#1_i\}_{i \in \fld{I}}}
\newcommand{\SRB}[2]{\{#1_i\}_{i=#2}^\infty}
\newcommand{\CA}{($\op{C}$,$\op{A}$)}
\newcommand{\AB}{($\op{A}$,$\op{B}$)}
\newcommand{\SG}[2]{\fld{#1}_{\op{#2}}}
\newcommand{\dom}[1]{D(\op{#1})}
\newcommand{\ro}[1]{\rho(\op{#1})}
\newcommand{\roi}[1]{\rho_\infty(\op{#1})}
\DeclareMathOperator{\ima}{Im}
\newcommand{\eigssp}[1]{\overline{\spanset{\ssp{E}_i}_{i\in {#1}}}}
\newcommand{\mr}[1]{\mathrm{#1}}
\newcommand{\FrOp}[2]{\underline{\op{#1}}(\ssp{#2})}

\def\QEDclosed{\hfill\IEEEQEDclosed}
\renewcommand{\qed}{\QEDclosed}
\renewenvironment{proof}[1][\proofname]{\noindent\nobreakspace{\bfseries #1}:\;}{\qed\par}

\newtheorem{theorem}{Theorem}
\newtheorem{lemma}{Lemma}
\newtheorem{remark}{Remark}
\newtheorem{definition}{Definition}

\newtheorem{corollary}{Corollary}

\title{\LARGE \bf
Invariant Subspaces of Riesz Spectral Systems with Application to Fault Detection and Isolation*
}

\author{Amir Baniamerian$^{1}$,  Nader Meskin$^{2}$ and Khashayar Khorasani$^{1}$
\thanks{*This publication was made possible by NPRP grant No. 4-195-2-065 from the
Qatar National Research Fund (a member of Qatar Foundation). The
statements made herein are solely the responsibility of the authors.}
\thanks{$^{1}$A. Baniamerian and K. Khorasani are with the Department of Electrical and Computer Engineering,
        Concordia University, Quebec, Canada,
        {\tt\small am\_bani@encs.concordia.ca} and {\tt\small kash@ece.concordia.ca.}}%
\thanks{$^{2}$N. Meskin is with the Department of Electrical Engineering, Qatar University,
        Doha, Qatar
        {\tt\small nader.meskin@qu.edu.qa. }}%
}

\begin{document}

\maketitle

\begin{abstract}
     A large class of hyperbolic and parabolic partial differential equation (PDE) systems, such as reaction-diffusion processes, when expressed in the  infinite-dimensional  (Inf-D) framework can be represented as  Riesz spectral (RS) systems. Compared to the finite dimensional (Fin-D) systems, the geometric theory of Inf-D systems for addressing certain fundamental control problems, such as  disturbance decoupling and fault detection and isolation (FDI), is rather quite limited due to complexity and existence of various types of invariant subspaces notions. Interestingly enough, these invariant concepts are  equivalent for Fin-D systems, although they are different in Inf-D representation. In this work, first  equivalence of various types of invariant subspaces that are defined for RS systems are investigated. This enables one to define and specify  the unobservability subspace for  RS systems. Specifically, necessary and sufficient conditions are derived for equivalence of various types of conditioned invariant subspaces. Moreover, by using duality properties,  various controlled invariant subspaces are developed. It is then shown that finite-rankness of the output operator enables one to derive  algorithms for computing invariant subspaces that under certain conditions, and unlike methods in the literature, converge in a finite number of steps. A geometric  FDI methodology for RS systems is then developed  by invoking the introduced invariant subspaces. Finally, necessary and sufficient conditions for solvability of the FDI problem are provided and analyzed. 
\end{abstract}
\begin{IEEEkeywords}
	Riesz spectral (RS) systems, infinite dimensional systems, fault detection and isolation, geometric control approach.
\end{IEEEkeywords}

\section{Introduction}\label{Sec:Intro}
The fault detection and isolation (FDI) problem of dynamical systems has  increasingly attracted  interest of researchers during the past two decades \cite{Isermann_Book,Chen_Book, Frank_FDI}. Advances in  control theory have led to development of various  capabilities for  control of quite complex dynamical systems. Due to complexity of these controlled systems one has to investigate and develop  more sophisticated FDI strategies and methodologies  \cite{Isermann_Book}.

A broad class of dynamical systems, ranging from chemical processes in the petroleum industry
to  heat transfer and compression processes in gas turbine engines, are represented by a
set of partial differential equations (PDEs). A large class of  hyperbolic and parabolic PDE systems can be represented and formulated as Riesz Spectral (RS) systems in an infinite dimensional (Inf-D) Hilbert space \cite{Curtainspectral}. The mathematical control theory of systems
governed by PDEs has seen a considerable progress in the past four decades \cite{Krstic_Book2010, Christofides_FaultControl, Lion_Book}. The control
theory of PDEs has been extended from ordinary differential equations (ODEs) by generally invoking  two
methodologies. The first is developed  through  approximation methods and the second  through exact methods. In the former
approach, one first approximates the original PDE by an ODE system (using for example finite element or finite difference methods), and then applies the established control theory
of ODEs to the approximated PDE model \cite{ACC2012, ACC2013, Armo_Demetrio}. In contrast, the latter or the exact approach tackles the PDE system holistically and without invoking any approximation \cite{Curtain_Book, Pazy_Book}.

Through application of  approximate methodologies, the FDI problem of PDEs and Inf-D systems has been investigated in the literature in e.g. \cite{ACC2012, Armo_Demetrio, Davis_Jor} and \cite{Davis_Thesis}. In \cite{ACC2012}, by using a geometric control approach, the FDI problem of a quasilinear parabolic PDE system is addressed. A Lyapunov-based method is proposed in \cite{Armo_Demetrio} for FDI of  a class of parabolic PDEs. However, given that in the above work the error dynamics analysis is based on the singular perturbation theory, only sufficient conditions for solvability of the FDI problem are provided in \cite{ACC2012, Armo_Demetrio, Davis_Jor}.

By using an array of sensors, the FDI problem of a beam structure has been investigated in \cite{SpeyerBeam}. In \cite{ACC2013}, by applying a finite difference method, a hyperbolic PDE is first approximated by a 2D Roesser model, and a geometric FDI approach is then developed.  Finally, the FDI problem of Inf-D systems is investigated in \cite{Demetrio, Demetriou_Adap_Jornal2007, Demetriou_Paper2002} by using exact methods, where an adaptive parameter estimation scheme is used to detect and estimate the fault severity.

The geometric theory of finite dimensional (Fin-D) linear systems was introduced in \cite{Wonham_Book,Basile_Book, Mass_Thesis, Massoumnia_Alone}, where fundamental problems such as disturbance decoupling and FDI problems have been  addressed. The geometric FDI approach has been extended to affine nonlinear systems in \cite{IsidoriCodis, IsidoriFDI}. The FDI problem of  Markovian jump linear systems is investigated in \cite{DrMeskin_Markov_TAC,DrMeskin_Markov}. By applying a discrete event-based FDI logic,  geometric FDI approaches for linear and nonlinear systems  have been extended in \cite{MKR_Nonlin2010} and \cite{MKR_SysTech2010}. Also, in \cite{SpeyerGeo} the geometric FDI approach is equipped with an $H_\infty$ method to enhance the robustness of the detection filters with respect to disturbance and noise signals. However, the geometric FDI approach has \emph{{not}} yet been investigated for Inf-D linear systems in general, and RS systems in particular. In this work, we develop for the \emph{{first time in the literature}} a geometric FDI methodology for RS systems.

In this work, we consider certain invariant subspaces, such as the $\op{A}$-invariant and conditioned invariant subspaces for  RS systems. For Inf-D systems, there are various definitions for $\op{A}$-invariant and conditioned invariant subspaces that are all equivalent in Fin-D systems. Therefore, in this work first  necessary and sufficient conditions for equivalence of various  conditioned invariant subspaces are formally shown for regular RS systems (this is specified formally in the next section). This result plays a crucial role subsequently in solvability of the FDI problem.  Next, by introducing an unobservability subspace we formulate the FDI problem in a geometric framework, and derive necessary and sufficient conditions for solvability of the problem. By utilizing duality notions, necessary and sufficient conditions for equivalence of controlled invariant subspaces are also obtained and derived.

It should be pointed out that in  \cite{ECC2014} we considered real diagonalizable RS systems. In this paper, we investigate invariant subspaces in more detail and derive the results  for more general class of RS systems as compared to those considered in \cite{ECC2014}. More specifically, the RS operator that is considered in this paper can have complex and finitely many multiple eigenvalues. Moreover, the FDI problem for only a diagonal RS system was introduced in \cite{ECC2014}, whereas in this paper, we derive necessary and sufficient
conditions for solvability of the FDI problem for a more general class of RS systems.

As shown in \cite{Curtain_Invariance_1986,Pandolfi_Disturbance_86,Zwart_Book}, for a general Inf-D system, the algorithms that are used to compute  invariant subspaces \underline{do not} converge in a finite number of steps. However, as we shall see subsequently, by using the results that are obtained in Section \ref{Sec:InvSpace} and under certain conditions one can compute invariant subspaces of regular RS systems in a finite number of steps. Specifically, we develop two schemes that converge in a finite number of steps for computing the conditioned invariant and unobservability subspaces.

To summarize, and in view of the above discussion the {main contributions} of this paper, and
 {all} developed for the \emph{first time in the literature}, can be listed as follows:
\begin{enumerate}
	\item Necessary and sufficient conditions for equivalence of various conditioned invariant subspaces for RS systems are obtained and analyzed. In the literature, {only} sufficient conditions for
equivalence of conditioned invariant subspaces of multi-input multi-output Inf-D systems are given. However, in this work we provide a \emph{single} necessary and sufficient condition.
	\item By using duality properties, necessary and sufficient conditions for equivalence of various controlled invariant subspaces are provided.
	\item The unobservability subspace for RS systems is introduced, and algorithms for computing this subspace that converge in a finite number of steps are proposed and derived.
	\item By taking advantage of the introduced subspaces, the FDI problem of RS systems is formulated and necessary and sufficient conditions for solvability of the FDI problem are developed and provided.
\end{enumerate}

The remainder of this paper is organized as follows. In Section \ref{Sec:Back}, RS systems are reviewed. Invariant subspaces are introduced, developed, and analyzed in Section \ref{Sec:InvSpace}. In Section \ref{Sec:FDI}, the FDI problem is formulated and necessary and sufficient conditions for  its solvability are provided. A numerical example is provided in Section \ref{Sec:SimResult} to demonstrate the capability of our proposed strategy.
Finally, Section \ref{Sec:Conclusion} provides the conclusions.\\
{\bf Notation:} The subspaces (finite and infinite dimensional) are denoted by $\ssp{A}$, $\ssp{B}$, $\cdots$. The notations $\overline{\ssp{V}}$ and $\ssp{V}^\bot$ denote the closure and orthogonal complement of the subspace $\ssp{V}$, respectively. We use the notation $\ssp{V}_1\perp\ssp{V}_2$ when every vector of $\ssp{V}_1$ is orthogonal to all the vectors of $\ssp{V}_2$. Without any confusion we use the notation $\overline{\lambda}$ to denote the conjugate of a complex number $\lambda$. The set of positive integers, complex, and real numbers are designated by $\fld{N}$, $\fld{C}$, and $\fld{R}$, respectively. The notation $\underline{\fld{N}}$ denotes the set $\fld{N}\cup\{0\}$. Consider a real subspace $\ssp{V}=\overline{\spanset{x_i}_{i\in\fld{I}}}$ ($\fld{I}\subseteq\fld{N}$). The corresponding complex subspace $\ssp{V}_{\fld{C}}$ is defined as all vectors $z$ that can be expressed as $z=\sum_{i\in\fld{I}}\zeta_ix_i$, where $\zeta_i\in\fld{C}$. The maps between two Fin-D vector spaces are designated by $A$, $B$, $\cdots$. The notations $\op{A}$, $\op{B}$, $\cdots$ denote the maps between two vector spaces such that at least one of them is an Inf-D vector space. Specifically,  we use the notations $I$ and $\op{I}$ to denote the identity operator on the Fin-D and Inf-D vector spaces, respectively. $\op{L}(\op{X})$ denotes the set of all bounded operators defined on $\op{X}$. The domain of an unbounded operator $\op{A}$ is denoted by $\dom{A}$. 
The operator of strongly continuous ($C_0$) semigroup  that is generated by $\op{A}$ is denoted by $\SG{T}{A}$.
The term $\ro{A}$ denotes the resolvent set of the operator $\op{A}$ (that is, all $\lambda\in\fld{C}$ such that $(\lambda\op{I} -\op{A})^{-1}$ exists and is a bounded operator). The set of all eigenvalues of $\op{A}$ is designated by $\sigma(\op{A})$. The largest real interval $[r,\infty)\subseteq\ro{A}$ is denoted by $\roi{A}$. The other notations are defined within the text of the paper.

\section{Background}\label{Sec:Back}
In this section, we review some of the basic concepts that are associated with a class of RS systems that will be investigated and further studied in detail in this paper.
\subsection{The Riesz Spectral (RS) Systems}
Consider the following infinite dimensional (Inf-D) system
\begin{equation}\label{Eq:GeneralID}
\begin{split}
\dot{x}(t)&=\op{A}x(t)+\op{B}u(t),\; x(0) = x_0,\\
y(t) &= \op{C}x(t),
\end{split}
\end{equation}
where $x(t)\in\op{X}$, $u(t)\in\fld{R}^m$ and $y(t)\in\fld{R}^q$ denote the state, input and output vectors, respectively,  and $\op{X}$ is a real  Inf-D separable Hilbert space
equipped with the dot-product $<\cdot,\cdot>$. Moreover, we consider the following {finite rank} output operator
\begin{equation}\label{Eq:OuputOperator}
\op{C}=\bbm <c_1,\cdot>, <c_2,\cdot>, \cdots, <c_q,\cdot>\ebm^\tran,
\end{equation}
and the finite rank operator $\op{B}$ is defined as $\op{B}=\sum_{i=1}^mb_iu_i$, where $b_i\in\op{X}$ and $u = [u_1,\cdots,u_m]^\tran$.


Moreover, we assume that the model \eqref{Eq:GeneralID} represents a well-posed system. This implies that the solution of system \eqref{Eq:GeneralID} is continuous with respect to the initial conditions for all $u(t)\in\fld{R}^m$ \cite{Curtain_Book}. This assumption is equivalent to stating that $\op{A}$ is closed and the infinitesimal generator of a strongly continuous ($C_0$) semigroup $\SG{T}{A}(t)$ is uniquely defined by $\op{A}$.
A $C_0$ semigroup $\fld{T}:\fld{R}^+\rightarrow\op{L}(\op{X})$ is the operator where the following conditions hold (\cite{Curtain_Book} Definition 2.1.2):
\begin{itemize}
	\item $\fld{T}(t+s) = \fld{T}(t)\fld{T}(s)\;$ for all $t,s\ge 0$.
	\item $\fld{T}(0) = \op{I}$.
	\item If $t\rightarrow 0^+$, then $||\fld{T}(t)x-x||\rightarrow 0\;$ for all $x\in\op{X}$.
\end{itemize}

Note that the solution of  system \eqref{Eq:GeneralID}  is given by $x(t) = \SG{T}{A}(t)x_0 + \int_0^t \SG{T}{A}(t-s)\op{B}u(s)\mathrm{d}s$ \cite{Curtain_Book}, where $x_0\in \op{X}$ denotes the initial condition. The following definitions are crucial for specifying the target system that is considered in this paper.
\begin{definition}\label{Def:RieszBasis}(\cite{Curtain_Book} - Definition 2.3.1)
	The set of vectors $\RieszBas{\phi}$, $\fld{I}\subseteq\fld{N}$ is called the Riesz basis for the Hilbert space $\op{X}$ if
	\begin{itemize}
		\item $\overline{\spanset{\phi_i}_{i\in\fld{I}}} = \op{X}$.
		\item There exist two positive numbers $M_1$ and $M_2$ (independent of $n$) such that for any $n\in\fld{N}$, we have $M_1\sum_{k=1}^n |\alpha_k|^2\leq ||\sum_{k=1}^n \alpha_k \phi_i||^2\leq M_2\sum_{k=1}^n |\alpha_k|^2$,  where $||\cdot||$ denotes the norm induced from $<\cdot,\cdot>$ and $\alpha_k\in\fld{R}$, $k=1\cdots,n$.
	\end{itemize}
\end{definition}
It can be shown (\cite{Curtain_Book}, Section 2.3) that if $\RieszBas{\phi}$ is a Riesz basis for $\op{X}$, then there exists a set of vectors $\RieszBas{\psi}$ such that $\psi_i\in\op{X}$ and $<\psi_i,\phi_k>=\delta_{ik}$ ($\delta_{ik}$ denotes the Dirac delta function), for all $i,k\in\fld{I}$. In other words, $\psi_i$'s and $\phi_k$'s are biorthonormal vectors \cite{Curtain_Book}.
The following lemma provides an important feature and property  of the Riesz basis.
\begin{lemma}(\cite{Curtain_Book}, Lemma 2.3.2-b)\label{Lem:RSUniqueRepresentation}
	Consider the Riesz basis  $\RieszBas{\phi}$ of the Hilbert space $\op{X}$. Then every $z\in\op{X}$ can be {uniquely} represented as $z=\sum_{i\in\fld{I}}<z,\psi_i>\phi_i$.
\end{lemma}
To define a regular RS operator, we need the following projection operator for each eigenvalue $\lambda_i$ of $\op{A}$ \cite{Schwartz1954}, namely
\begin{equation}\label{Eq:EigenPro}
\op{P}_i:\op{X}\rightarrow\op{X},\;\; \op{P}_i = \frac{1}{2\pi j}\int_{\Gamma_{i}} (\lambda\op{I}-\op{A})^{-1} \text{d}\lambda,
\end{equation}
where $i\in\fld{I}_\lambda$ ($\fld{I}_\lambda$ is an index set for $\sigma(\op{A})$), $\Gamma_i$ is a simple closed curve surrounding only the eigenvalue $\lambda_i$. This represents the projection on the subspace of generalized eigenvectors of $\op{A}$ corresponding to $\lambda_i$, that is, the subspace spanned by all $\phi_i$'s satisfying $(\lambda_i\op{I}-\op{A})^n\phi_i=0$, for some positive integer $n$.
\begin{definition}\label{Def:RegularRS}\cite{Schwartz1954}
	The operator $\op{A}$ is called  a {regular}  RS operator, if
	\begin{enumerate}
			\item All but finitely many of the eigenvalues (with finite multiplicity) are simple.
			\item The (generalized) eigenvectors of the operator $\op{A}$, $\RieszBas{\phi}$, form a Riesz basis for $\op{X}$ (but defined on the field $\fld{C}$), and consequently,  ${\sum_{i\in\fld{I}_\lambda}\op{P}_i} = \op{I}$ (that is an identity operator on $\op{X}$).  
	\end{enumerate}
\end{definition}
\begin{remark}\label{Rem:CondOnDefRegRSOp}
	As we shall see subsequently, to derive a necessary condition for solvability of the FDI problem, it is necessary that a bounded perturbation of $\op{A}$ (that is, $\op{A+D}$ where $\op{D}$ is a bounded operator) is also a regular RS operator. This property  holds if $\sum_{i} \frac{1}{d_i^2}<\infty$, where $d_i = \inf_{\lambda\in\sigma(\op{A})-\{\lambda_i\}} |\lambda-\lambda_i|$ \cite{Schwartz1954} (Theorem 1). Therefore, in this paper it is assumed that the operator $\op{A}$ satisfies the above condition. It should be pointed out that a large class of RS systems, including  discrete RS systems satisfy this condition \cite{ZwartGeneralEig2001}.\qed
\end{remark}

If the operator $\op{A}$ in the system \eqref{Eq:GeneralID} is a regular RS operator and the operators $\op{B}$ and $\op{C}$ are bounded and finite rank we designate the system \eqref{Eq:GeneralID} as a regular RS system. Moreover, the system \eqref{Eq:GeneralID} is well-posed if and only if $ \underset{\lambda_i\in\sigma(\op{A})}{\text{sup}} \lambda_i <\infty$ (this  is a feasible assumption from the applications point of view)\cite{Chen_Book}. Also, according to the Definitions \ref{Def:RieszBasis} and \ref{Def:RegularRS}, one can show that \cite{ZwartGeneralEig2001}
\begin{equation}\label{Eq:A_TDef}
\begin{split}
\op{A} &= \sum_i\lambda_i\sum_{k=1}^{n_i}<\cdot,\psi_{i,k}>\phi_{i,k},\;\;
\end{split}
\end{equation}
where $n_i$ denotes the number of (generalized) eigenvectors corresponding to the eigenvalues $\lambda_i$ (if $\lambda_i$ is a distinct eigenvalue then $n_i=1$, and if $\lambda_i$ is repeated we have $n_i>1$). Also, $\phi_{i,k}$'s and $\psi_{i,k}$'s are the (generalized) eigenvectors and the corresponding biorthonormal vectors of  $\lambda_i$, respectively.

Given that we are interested in RS systems that are defined on the field $\fld{R}$, we need to work with eigenspaces instead of eigenvectors (eigenvalues and eigenvectors in \eqref{Eq:A_TDef} can be complex). If an eigenvalue is real, the corresponding eigenspace is equal to $\op{P}_i\op{X}$, where $\op{P}_i$ is the corresponding projection that is defined in \eqref{Eq:EigenPro}. Let $\lambda=a + jb$ and $\overline{\lambda} = a-jb$ be a pair of complex conjugate eigenvalues of $\op{A}$. Since $\op{A}$ is a real operator, it is easy to show that if  $\phi = v_1+ jv_2$ is a (generalized) eigenvector corresponding to $\lambda$, then $\overline{\phi} = v_1- jv_2$ is a (generalized) eigenvector corresponding to $\overline{\lambda}$ (the conjugate of $\lambda$). The corresponding \emph{real} eigenspace to $\lambda$ and $\overline{\lambda}$ is constructed by  $\mr{span}\{v_1^i,v_2^i\}_{i=1}^{n}$, where $v_1^i\pm jv_2^i$ correspond to the (generalized) eigenvectors of $\op{A}$, and $n$ denotes the algebraic multiplicity of $\lambda$. We denote the real eigenspace of $\op{A}$ corresponding to $\lambda_i$ by $\ssp{P}_i$. It should be pointed out that $\dim(\ssp{P}_i) = n_i$ and $\dim(\ssp{P}_i) = 2n_i$ for real and complex eigenvalue $\lambda_i$, respectively (where $n_i$ is the algebraic multiplicity of $\lambda_i$). Note that Condition 2 in Definition \ref{Def:RegularRS} implies that $\overline{\sum_{i\in\fld{I}_\lambda}\ssp{P}_i} = \op{X}$ (defined on $\fld{R}$). Also,  we have  $\ssp{P}_i\subseteq\dom{A}$ and $\op{A}\ssp{P}_i\subseteq\ssp{P}_i$. Moreover, we designate the subspace $\ssp{E}_i\subseteq\ssp{P}_i$ as a \underline{sub-eigenspace} if $\op{A}\ssp{E}_i\subseteq\ssp{E}_i$.

 \begin{remark}\label{Rem:SubEig_SimpleEig}
 It is worth noting that  the only proper sub-eigenspace of an eigensapce corresponding to a simple eigenvalue  is $0$. In other words, let $\ssp{P}$ be an eigenspace corresponding to a simple eigenvalue $\lambda_0$. If $\ssp{E}\subset\ssp{P}_0$ (and $\ssp{E}\neq\ssp{P}_0$), then $\op{A}\ssp{E}\subseteq\ssp{E}$ implies $\ssp{E}=0$.
 \end{remark}

\section{Invariant Subspaces}\label{Sec:InvSpace}
Invariant subspaces play a prominent role in the geometric control theory of dynamical systems \cite{Wonham_Book,DynFeedbackInf_Book,Massoumnia_Alone,Zwart_Book}. For the FDI problem (which is formally defined in Section \ref{Sec:FDI}), one requires to work with three invariant subspaces, namely $\op{A}$-invariant, conditioned invariant, and unobservability subspaces. To investigate the disturbance decoupling problem (refer to \cite{Wonham_Book} for more detail), one deals with controlled invariant and controllability subspaces that are dual  to conditioned invariant and unobservability subspaces, respectively \cite{Mass_Thesis}.

In the literature,  $\op{A}$-invariant and conditioned invariant subspaces have been introduced for Inf-D systems \cite{DynFeedbackInf_Book, Curtain_Invariance_1986, Pandolfi_Disturbance_86,Curtainspectral}. Due to complexity of Inf-D systems, various kinds of invariant subspaces are available (although these are all equivalent in Fin-D systems). The necessary and sufficient conditions for equivalence of $\op{A}$-invariant subspaces have been obtained in the literature \cite{Curtain_Book}. However, for equivalence of conditioned invariant subspaces, the  results that are available are {only} limited to sufficient conditions. In the following subsections, we first review invariant subspaces and provide necessary and sufficient conditions for equivalence of conditioned invariant subspaces for regular RS systems.  Then, by invoking  duality properties,  necessary and sufficient conditions for equivalence of controlled invariant subspace are shown formally. Moreover, an unobservability subspace for RS systems is also introduced.

Generally, for Inf-D systems the algorithms that are developed to compute  invariant subspaces require an {infinite} number of steps to converge. In this section, it is shown that the finite-rankness of the output operator enables us, for the \textit{first time} in the literature, to develop algorithms for computing  conditioned invariant and unobservability subspaces that converge in a {finite} number of steps.
\subsection{$\op{A}$-Invariant Subspace}
There are two different definitions that are related to the $\op{A}$-invariance property. Unlike Fin-D systems, these definitions are not equivalent for Inf-D systems.  In this subsection, we review these definitions and investigate various types of unobservable subspaces for the RS system \eqref{Eq:GeneralID}.
\begin{definition}\label{Def:AInv} \cite{DynFeedbackInf_Book}
	\begin{enumerate}
		\item The closed subspace $\ssp{V}\subseteq\op{X}$ is called $\op{A}$-invariant if $\op{A}(\ssp{V}\cap D(\op{A}))\subseteq \ssp{V}$.
		\item The closed subspace $\ssp{V}\subseteq\op{X}$ is $\SG{T}{A}$-invariant if $\SG{T}{A}(t)\ssp{V}\subseteq\ssp{V}$ for all $t\in[0,\infty)$, where $\SG{T}{A}$ denotes the $C_0$ semigroup generated by $\op{A}$.
	\end{enumerate}
\end{definition}

For the Fin-D systems, items 1) and 2) in the above definition are equivalent, however for Inf-D systems, item 2) is stronger than item 1). In other words, every $\SG{T}{A}$-invariant subspace is $\op{A}$-invariant, however the reverse is not valid in general  \cite{DynFeedbackInf_Book}. In the geometric control theory of dynamical systems, one needs subspaces that are $\SG{T}{A}$-invariant. Since dealing with  $\SG{T}{A}$-invariant subspaces is more challenging than $\op{A}$-invariant subspaces, we are interested in cases where they are equivalent. For a general Inf-D system, a sufficient condition to have this equivalence is $\ssp{V}\subseteq D(\op{A})$ \cite{DynFeedbackInf_Book}, which is quite a restricted and limited condition. However, the following lemma provides necessary and sufficient conditions for $\fld{T}_\op{A}$-invariance property.
\begin{lemma}\cite{Curtain_Book} (Lemma 2.5.6)\label{Lem:A-1_T-Inv}
	Consider an infinitesimal generator $\op{A}$ (more general than RS operators), and its corresponding $\SG{T}{A}$ operator and a closed subspace $\ssp{V}$. Then $\ssp{V}$ is $\SG{T}{A}$-invariant if and only if $\ssp{V}$ is $(\lambda\op{I} -\op{A})^{-1}$-invariant, where $\lambda\in\roi{A}$. 
\end{lemma}
Another important result on $\SG{T}{A}$-invariant subspaces for a regular RS system that is provided in \cite{Zwart_Book} (Theorem IV.6) is given next.
\begin{lemma}\label{Lem:EqualInvSpace_Zwart}\cite{Zwart_Book}
	Consider the Inf-D system \eqref{Eq:GeneralID}, where $\op{A}$ is a regular RS operator and the $\op{A}$-invariant subspace is denoted by $\ssp{V}$. Then $\ssp{V}$ is $\SG{T}{A}$-invariant if and only if $\ssp{V}=\overline{\spanset{\ssp{D}_i}_{i\in\fld{I}_1}}$, where $\fld{I}_1\subseteq\fld{N}$ and $\ssp{D}_i\subseteq\op{P}_i\op{X}$, is $\op{A}$-invariant. 
\end{lemma}

As stated in the preceding section, the  eigenvalues (and the  corresponding eigenvectors) of $\op{A}$ may be complex, and Lemma \ref{Lem:EqualInvSpace_Zwart} is provided for complex subspaces. However, for  geometric control approach one needs to work with real subspaces. The following  corollary provides the necessary and sufficient conditions for equivalence of Definition \ref{Def:AInv}, items 1) and 2) for  regular RS systems and real subspaces.

\begin{corollary}\label{Col:A-T-Inv}
	Consider the regular RS system \eqref{Eq:GeneralID} and the $\op{A}$-invariant subspace $\ssp{V}$. The real subspace $\ssp{V}$ is $\SG{T}{A}$-invariant if and only if $\ssp{V}=\overline{\spanset{\ssp{E}_i}_{i\in\fld{I}_1}}$, where $\ssp{E}_i$'s denote the sub-eigenspaces of $\op{A}$  and $\fld{I}_1\subseteq\fld{N}$.
\end{corollary}
\begin{proof}
	Let $\phi^k=v_1^k+jv_2^k, k = 1,\cdots,n_i$ denote the corresponding (generalized) eigenvectors for the eigenvalue $\lambda_i = \gamma_1+j\gamma_2$ of $\op{A}$, where  $n_i$ denotes the algebraic multiplicity of $\lambda_i$, and  $\gamma_\ell$ and $v_\ell^k$ (for $\ell=1,2$) are real numbers and vectors, respectively. Since $\op{A}$ is a regular RS operator, it follows that the eigenspace corresponding to $\lambda_i$ (and its conjugate) is equal to $\spanset{v_1^k,v_2^k}_{k=1}^{n_i}$.\\
	({\bf If part}): Let $\ssp{V}=\overline{\spanset{\ssp{E}_i}_{i\in\fld{I}_1}}$. The corresponding complex subspace of $\ssp{V}$ (refer to the Notation description in Section \ref{Sec:Intro}) is then expressed by $\ssp{V}_\fld{C} = \overline{\spanset{\ssp{D}_i}_{i\in\fld{I}_1}}$, where $\ssp{D}_i$ (and its conjugate) is the corresponding complex subspace to $\ssp{E}_i$. Consequently,  $\ssp{V}_\fld{C}$ is $\op{A}$-invariant. By Lemma \ref{Lem:EqualInvSpace_Zwart}, $\ssp{V}_\fld{C}$ is $\SG{T}{A}$-invariant. Hence, $\SG{T}{A}(t)(v_1+jv_2)\in\ssp{V}_\fld{C}$, for all $v_1+jv_2\in\ssp{V}_\fld{C}$ and $t\ge 0$. Since $\op{A}$ and  $\SG{T}{A}$ are real, by referring to the definition of $\ssp{V}_\fld{C}$ we have $v_1,v_2\in\ssp{V}$ and $\SG{T}{A}(t)v_1, \SG{T}{A}(t)v_2\in\ssp{V}$ for all $t\ge 0$. Therefore, $\SG{T}{A}(t)\ssp{E}_i\subseteq\ssp{E}_i$ implying that $\ssp{V}$ is $\SG{T}{A}$-invariant.\\
	({\bf Only if part}): Let $\ssp{V}$ be $\SG{T}{A}$-invariant. The corresponding complex subspace $\ssp{V}_\fld{C}$  is also $\SG{T}{A}$-invariant. Again, by using Lemma \ref{Lem:EqualInvSpace_Zwart},  $\ssp{V}_\fld{C} = \overline{\spanset{\phi_i}_{i\in\fld{I}_1}}$. Therefore, $\ssp{V}=\overline{\spanset{\ssp{E}_i}_{i\in\fld{I}_1}}$.
	This completes the proof of the corollary.
\end{proof}	

In this work, we are mainly concerned with two important invariant subspaces of RS systems as discussed below. We denote the largest $\op{A}$- and  $\SG{T}{A}$-invariant subspaces that are contained in $\ssp{C}$ by $<\ssp{C}|\op{A}>$ and $<\ssp{C}|\SG{T}{A}>$, respectively. The $\op{A}$-unobservable subspace of the system \eqref{Eq:GeneralID} is defined by $\ssp{N}_\op{A} = <\ker\op{C}|\op{A}> = \bigcap_{n\in\underline{\fld{N}}}\ker\op{C}\op{A}^n$. Also, the unobservable subspace of the system \eqref{Eq:GeneralID} is defined by  $\ssp{N} = <\ker \op{C}|\SG{T}{A}> = \bigcap_{t\geq0}\ker \op{C}\SG{T}{A}(t)$ \cite{Curtain_Invariance_1986}. Note that $\ssp{N}_\op{A}\subseteq D(\op{A}^n)$ for all $n\in\fld{N}$ and is not necessarily $\SG{T}{A}$-invariant. However, as shown subsequently, by using this subspace one is enabled to develop an algorithm to compute the conditioned invariant subspaces in a {finite} number of steps. Moreover, these subspaces will be used in Section \ref{Sec:UnObse} to introduce the unobservability subspace of RS systems, where the following corollary plays a crucial role.
\begin{corollary}\label{Col:UnObs_Span}
	Consider the RS system \eqref{Eq:GeneralID}, where $\op{A}$ is a regular RS operator with a bounded output operator $\op{C}$. The unobservable subspace $\ssp{N}$ is the largest subspace contained in $\ker\op{C}$ that can be expressed as $\overline{\spanset{\ssp{E}_i}_{i\in\fld{I}}}$, where $\ssp{E}_i$'s are sub-eigenspaces of $\op{A}$ and $\fld{I}\subseteq\fld{N}$.
\end{corollary}
\begin{proof}
	As stated above, $\ssp{N}$ is $\SG{T}{A}$-invariant, and consequently by using Corollary \ref{Col:A-T-Inv}, $\ssp{N} = \overline{\spanset{\ssp{E}_i}_{i\in\fld{I}}}$. Moreover, since $\ssp{N}$ is the largest $\SG{T}{A}$-invariant that is contained in $\ker\op{C}$ \cite{Curtain_Invariance_1986}, the result follows readily. This completes the proof of the corollary.
\end{proof}
\subsection{Conditioned Invariant Subspaces}
In this subsection, the conditioned invariant subspaces of the system \eqref{Eq:GeneralID} are defined and characterized. Not surprisingly, various definitions, that are all equivalent in  Fin-D systems, are available for conditioned invariant subspaces of Inf-D systems that are \underline{not}  equivalent to one another \cite{Curtain_Invariance_1986}. This subsection mainly concentrates on deriving necessary and sufficient conditions where these definitions are shown to be equivalent.  Let us first define the notion of conditioned invariant subspace.
\begin{definition}\label{Def:ConditionInv}\cite{Curtain_Invariance_1986}
	\begin{enumerate}
		\item  The closed subspace $\ssp{W}\subseteq\op{X}$ is designated as \CA-invariant if $\op{A}(\ssp{W}\cap D(\op{A})\cap\ker \op{C})\subseteq \ssp{W}$.
		\item The closed subspace $\ssp{W}\subseteq\op{X}$ is feedback \CA-invariant if there exists a bounded operator $\op{D}:\fld{R}^q\rightarrow\op{X}$ such that $\ssp{W}$ is invariant with respect to $(\op{A+DC})$, as per Definition \ref{Def:AInv}, item 1).
		\item The closed subspace $\ssp{W}\subseteq\op{X}$ is $\fld{T}$-conditioned invariant if there exists a bounded operator $\op{D}:\fld{R}^q\rightarrow\op{X}$ such that (i) the operator $(\op{A+DC})$ is the infinitesimal generator of a $C_0$-semigroup $\fld{T}_\op{A+DC}$; and (ii) $\ssp{W}$ is invariant with respect to $\fld{T}_\op{A+DC}$, as per Definition \ref{Def:AInv}, item 2).
	\end{enumerate}
\end{definition}
It should be pointed out that in the literature $\fld{T}$-conditioned invariant is also called $\fld{T}(\op{C},\op{A})$-invariant \cite{Curtain_Invariance_1986}. It can be shown that Definition \ref{Def:ConditionInv}, item 3) $\Rightarrow$ item 2) $\Rightarrow$ item 1) \cite{Curtain_Invariance_1986}. 
A sufficient condition for equivalence of the above definitions is developed in \cite{Curtain_Invariance_1986}.
\begin{lemma}\label{Lem:CAInv}\cite{Curtain_Invariance_1986}
	A given \CA-invariant subspace $\ssp{W}$ is $\fld{T}$-conditioned invariant, if $\op{C}\ssp{W}$ is closed and $\ssp{W}\subseteq D(\op{A})$.
\end{lemma}

In this subsection, we show that Definition \ref{Def:ConditionInv}, item 1) and item 2) are equivalent for the system \eqref{Eq:GeneralID}, when the finite rank output operator is represented by \eqref{Eq:OuputOperator} (even if $\ssp{W}\not\subset\dom{A}$). Moreover, we derive necessary and sufficient conditions for $\fld{T}$-conditioned invariance. These results enable one to subsequently derive the necessary  and sufficient conditions for solvability of the FDI problem. Towards this end, we first need the following lemma.
\begin{lemma}\label{Lem:SumClosedSpace}
	Consider the closed subspace $\ssp{V}= \overline{\spanset{x_i}_{i\in\fld{I}}}$, where $x_i\in\op{X}$ (and not necessarily orthogonal) and $\fld{I}\subseteq\fld{N}$. Then
	\begin{equation}
	\begin{split}
	\ssp{V} &= \overline{\ssp{V}_{\mathrm{inf}}+\ssp{V}_\mathrm{f}} = \ssp{V}_{\mathrm{inf}}+\ssp{V}_\mathrm{f},\;
	\end{split}
	\end{equation}
	where $\ssp{V}_\mathrm{f} = \spanset{x_i}_{i\in \fld{J}}$, $\ssp{V}_{\mathrm{inf}} = \overline{\spanset{x_i}_{i\in\fld{I}-\fld{J}}}$ and $\fld{J}$ is a finite subset of $\fld{I}$.
\end{lemma}
\begin{proof}
	It follows readily that $\spanset{x_i}_{i\in\fld{I}-\fld{J}}+\ssp{V}_\text{f}$ is dense in $\ssp{V}$. Hence, the subspace ${\ssp{V}_{\text{inf}}}+\ssp{V}_\text{f}$ is also dense in $\ssp{V}$. Furthermore, since $\ssp{V}_\text{f}$ is a Fin-D subspace, it is a closed subspace. Therefore, by using the Proposition 1.7.17 in \cite{BanachTheory_Book} (which states that the sum of two closed subspaces is also closed if at least one of them is Fin-D), it follows that ${\ssp{V}_{\text{inf}}}+\ssp{V}_\text{f}$ is closed. Since, ${\ssp{V}_{\text{inf}}}+\ssp{V}_\text{f}$ is closed and dense in $\ssp{V}$, we have ${\ssp{V}_{\text{inf}}}+\ssp{V}_\text{f}=\ssp{V}$. This completes the proof of the lemma.
\end{proof}
The following lemma shows the equivalence of \CA- and feedback \CA-invariance properties for a general Inf-D system provided that the output operator is a finite rank operator (as considered to be satisfied by the model \eqref{Eq:OuputOperator} in this paper).
\begin{lemma}\label{Lem:CA_equalency}
	Consider the Inf-D system \eqref{Eq:GeneralID}, where $\op{A}$ is the infinitesimal generator of a $C_0$ semigroup (more general than the regular RS operator) and the finite rank output operator is  given by \eqref{Eq:OuputOperator}. Let $\ssp{W}\subseteq\op{X}$ be a closed subspace such that $\overline{\dom{A}\cap\ssp{W}}=\ssp{W}$. The subspace $\ssp{W}$ is \CA-invariant if and only if it is feedback \CA-invariant.
\end{lemma}
\begin{proof}
	As pointed out earlier, every feedback \CA-invariant subspace is \CA-invariant. Therefore, we only show the converse.
	By definition, we have $\op{A}(\ssp{W}\cap\ker\op{C}\cap D(\op{A}))\subseteq \ssp{W}$. Since $\overline{\ssp{W}\cap\dom{A}} = \ssp{W}$, and $\ssp{W}$ is separable ($\ssp{W}$ is a closed subspace of the separable Hilbert space $\op{X}$), there exists a basis $\{w_i\}_{i\in\fld{I}}$ for $\ssp{W}$ such that $w_i\in\dom{A}$. Let us rearrange the basis $\{w_i\}_{i\in\fld{I}}$ such that the first $n_\mr{f}$ vectors construct the Fin-D subspace $\ssp{W}_\mr{f}=\spanset{w_i}_{i=1}^{n_\mr{f}}\subset\dom{A}$, where $\ssp{W}_\mr{f}\cap\ker\op{C}=0$ and $n_\mr{f}=\dim(\ssp{W}\cap(\ssp{W}\cap\ker\op{C})^\perp)$. It should be pointed out that from \eqref{Eq:OuputOperator} (i.e. the finite rankness of $\op{C}$) and the fact that $\ssp{W}_\mr{f}\cap\ker\op{C}=0$, it follows that $\dim(\ssp{W}_\text{f})=n_\mr{f}\leq q<\infty$. Note that if $n_\mr{f}=0$ it implies that $\ssp{W}\subseteq\ker\op{C}$, and therefore it is $\op{A}$-invariant and by setting $\op{D}=0$ it is also feedback \CA-invariant. Now, without loss of any generality we assume that $w_i\in\ker\op{C}$ for all $i>n_\mr{f}$ (if $w_i\notin\ker\op{C}$, one can remove the projection of $w_i$ on $\ssp{W}_\mr{f}$ and call it as $w_i^n\in\ker\op{C}$. Since $\ssp{W}_\mr{f}\subset\dom{A}$, it follows that $w_i^n\in\dom{A}$). Given that $\dim(\ssp{W}_\mr{f})<\infty$, now by using Lemma \ref{Lem:SumClosedSpace} one obtains $\ssp{W} = \ssp{W}_{\mr{inf}}+\ssp{W}_\mr{f}$, where $\ssp{W}_\mr{inf} = \ssp{W}\cap\ker\op{C}=\overline{\spanset{w_i}_{i>n_\mr{f}}}$.
	
	
	 We now show how one can construct a bounded operator $\op{D}$ such that $(\op{A+DC})(\ssp{W}\cap\dom{A})\subset\ssp{W}$. Let $\op{A}w_i = x_i\in\op{X}$, $i=1,\cdots,n_\mr{f}$. We construct $\op{D}$ such that $\op{D} \op{C}[w_1,\cdots,w_{n_\mr{f}}] = -[x_1,\cdots,x_{n_\mr{f}}]$. Note that $\ssp{W}_\text{f}\cap \ker \op{C} = 0$, $\dim(\ssp{W}_\text{f})<\infty$, and $\op{C}$ is a bounded operator. It follows that $\op{C}$ is an invertible operator from $\ssp{W}_\text{f}$ onto $\ssp{Y} = \op{C}\ssp{W}_\text{f}\subseteq\fld{R}^q$. In other words, $C_w=\op{C}|_{\ssp{W}_\mathrm{f}}:\ssp{W}_\text{f}\rightarrow\ssp{Y}$ is a bijective map. Therefore, $C_w = \op{C}[w_1,\cdots,w_{n_\text{f}}]$ is a monic matrix (i.e., $\ker C_w = 0$), and consequently always there is a solution for  $D_w:\ssp{Y}\rightarrow\op{X}_\mathrm{f}$, such that  $D_w C_w = -[x_1,\cdots,x_{n_\mr{f}}]$, where $\op{X}_\mr{f}=\spanset{x_i}_{i=1}^{n_\mr{f}}$. A solution to $\op{D}:\fld{R}^q\rightarrow\op{X}$ is an extension of $D_w$ as $\op{D}y = \op{Q}D_wy_\text{1}$, where $y\in\fld{R}^q$, $y=y_{\text{1}}+y_\text{2}$, $y_\text{1}\in\ssp{Y}$, $y_{2}\in\ssp{Y}^\perp$ and $\op{Q}$ is the embedding operator from $\op{X}_\mr{f}$ to $\op{X}$. Since $\ssp{Y}$ is Fin-D, it follows that $\op{D}$ is bounded. Now, set $x\in(\ssp{W}\cap\dom{A})$. Since $\ssp{W}_\mr{f}\subset\dom{A}$, one can write $x=x_\mr{inf}+x_\mr{f}$, where $x_\mr{inf}\in(\ssp{W}_\mr{inf}\cap\dom{A})$ and $x_\mr{f}\in\ssp{W}_\mr{f}$. Given that $\ssp{W}$ is \CA-invariant, it follows that $(\op{A+DC})x_\mr{inf} = \op{A}x_\mr{inf}\in\ssp{W}$, and by definition of $\op{D}$, we obtain $(\op{A+DC})x_\mr{f} = 0$. Therefore, $(\op{A+DC})x\in\ssp{W}$, and consequently $\ssp{W}$ is a feedback \CA-invariant subspace. This completes the proof of the lemma.
\end{proof}
As shown in \cite{Zwart_Book} the $\fld{T}$-conditioned invariance and \CA-invariance  are not generally equivalent. Moreover, if $\op{C}$ is not finite rank the feedback \CA-invariance and \CA-invariance are not equivalent \cite{Zwart_Book,Curtain_Invariance_1986}. However, Lemma \ref{Lem:CA_equalency} shows the equivalence between the feedback \CA-invariance and \CA-invariance in the sense of Definition 2, if the output operator $\op{C}$ is finite rank and $\overline{\ssp{W}\cap\dom{A}}=\ssp{W}$.

The following lemma shows that  the $\fld{T}$-conditioned invariance is an independent property from the bounded operator $\op{D}$. This result allows one to derive  necessary and sufficient conditions for the $\fld{T}$-conditioned invariance.
\begin{lemma}\label{Lem:DiffOp4CondInv}
	Consider a $\fld{T}$-conditioned invariant subspace $\ssp{W}$ such that $\fld{T}_{\op{A+D}_1\op{C}}\ssp{W}\subseteq\ssp{W}$, and consider a bounded operator $\op{D}_2$ such that $(\op{A+D}_2\op{C})(\ssp{W}\cap\dom{A})\subseteq\ssp{W}$. Then $\fld{T}_{\op{A+D}_2\op{C}}\ssp{W}\subseteq\ssp{W}$.
\end{lemma}
\begin{proof}
	By invoking Lemma \ref{Lem:A-1_T-Inv}, we have $(\lambda\op{I}-(\op{A+D}_1\op{C}))^{-1}\ssp{W}\subseteq\ssp{W}$, for all $\lambda\in\rho_\infty(\op{A+D}_1\op{C})$. Let us set $\lambda\in\rho_\infty(\op{A+D}_1\op{C})\cap\rho_\infty(\op{A+D}_2\op{C})$ (by using the Hille-Yosida theorem (\cite{Curtain_Book}-Theorem 2.1.12), where it is shown that for every infinitesimal generator $\op{A}$ there exists a real number $r\in\fld{R}$ such that $[r,\;\infty)\subset\roi{A}$ and we have the set $\rho_\infty(\op{A+D}_1\op{C})\cap\rho_\infty(\op{A+D}_2\op{C})$ non-empty). Based on results of Lemma \ref{Lem:A-1_T-Inv}, we need to show that $(\lambda\op{I-(A+D}_2\op{C}))^{-1}\ssp{W}\subseteq\ssp{W}$. First, let $\ssp{W}_c = \{y|y\in\ssp{W}\; ; \;(\lambda\op{I}-(\op{A+D}_1\op{C}))^{-1}y\in\ssp{W}_\mr{f}\}$, where $\ssp{W}_\mr{f}\subset\dom{A}$ is defined as in the proof of Lemma \ref{Lem:CA_equalency} and $\ssp{W}_\infty = \{y|y\in\ssp{W}\; ; \;(\lambda\op{I}-(\op{A+D}_1\op{C}))^{-1}y\in\ssp{W}\cap\ker\op{C}\}$. Since $\ssp{W}=\ssp{W}_\mr{f}+\ssp{W}\cap\ker\op{C}$, $\dim(\ssp{W}_\mr{f})<\infty$ and $(\lambda\op{I}-(\op{A+D}_1\op{C}))^{-1}$ is bounded and bijective, it then follows that $\ssp{W} = \ssp{W}_c+\ssp{W}_\infty$.
	 Let $y\in\ssp{W}_\infty$ and $x=(\lambda\op{I}-(\op{A+D}_1\op{C}))^{-1}y$.  Given that $x\in\ker\op{C}$, it follows that
	 \begin{align}\label{Eq:xWc}
	 y = (\lambda\op{I}-(\op{A+D}_1\op{C}))x =  (\lambda\op{I}-(\op{A+D}_2\op{C}))x =  (\lambda\op{I}-\op{A})x.
	 \end{align}
	 Since $\ssp{W}$ is $(\lambda\op{I}-(\op{A+D}_1\op{C}))^{-1}$-invariant, one obtains $x\in\ssp{W}$, and consequently we have $(\lambda\op{I}-(\op{A+D}_2\op{C}))^{-1}y=x\in\ssp{W}$.

	Next, by following along the steps provided below we show that if $y\in\ssp{W}_c$ then $(\lambda\op{I}-(\op{A+D}_2\op{C}))^{-1}y\in\ssp{W}$.
	\begin{enumerate}
		\item Let $\{w_i\}_{i=1}^{n_\mr{f}}$ be a basis of $\ssp{W}_\mr{f}$ and set $z_i = (\lambda\op{I}-(\op{A+D}_2\op{C}))w_i\in\ssp{W}$ for $i=1,\cdots,n_\mr{f}$ (as $(\op{A+D}_2\op{C})(\ssp{W}\cap\dom{A})\subseteq\ssp{W}$). Since $\ssp{W}=\ssp{W}_c+\ssp{W}_\infty$ one can write $z_i=z_c^i+z_\infty^i$, where $z_c^i\in\ssp{W}_c$ and $z_\infty^i\in\ssp{W}_\infty$.
		\item We show that $z_c^i$'s are linearly independent. Towards this end, assume $z_c^i$ are linearly dependent and therefore we obtain $\sum_{i=1}^{n_\mr{f}}\zeta_iz_c^i=0$, where $\zeta_i\in\fld{R}$ for $i=1,\cdots,n_\mr{f}$. Hence, one can write 	 $(\lambda\op{I}-(\op{A+D}_2\op{C}))w=z_\infty$, where $w =\sum_{i=1}^{n_\mr{f}}\zeta_iw_i\neq 0$ (since $w_i$'s are basis vectors), and $z_\infty=\sum_{i=1}^{n_\mr{f}}\zeta_iz_i=\sum_{i=1}^{n_\mr{f}}\zeta_iz_\infty^i\in\ssp{W}_\infty$. Consequently, given $w=(\lambda\op{I}-(\op{A+D}_2\op{C}))^{-1}z_\infty$ and by the definition of $\ssp{W}_\infty$ we have $w\in\ker\op{C}$ and $w = (\lambda\op{I}-\op{A+D}_1\op{C})^{-1}z_\infty$ \footnote{Since $z_\infty\in\ssp{W}_\infty$, we obtain $(\lambda\op{I}-(\op{A+D}_1\op{C}))^{-1}z_\infty\in\ker\op{C}$, and consequently $w =(\lambda\op{I}-(\op{A+D}_1\op{C}))^{-1} =(\lambda\op{I}-(\op{A+D}_2\op{C}))^{-1}z_\infty\in\ker\op{C}$.}. This is in contradiction with the fact $w\in\ssp{W}_\mr{f}$ (recall that $\ssp{W}_\mr{f}\cap\ker\op{C}=0$). Therefore, $z_c^i$'s are linearly independent. Since the resolvent operators are bijective and $\ssp{W}_c$ is Fin-D, we obtain $\dim(\ssp{W}_c)=\dim(\ssp{W}_\mr{f})=n_\mr{f}$, and consequently $\{z_c^i\}_{i=1}^{n_\mr{f}}$ is a basis of $\ssp{W}_c$.
		\item We show that $(\lambda\op{I}-(\op{A+D}_2\op{C}))^{-1}z_c^i\in\ssp{W}$, where $z_i=z_c^i+z_\infty^i=(\lambda\op{I}-\op{A+D}_2\op{C})w_i$, $w_i\in\ssp{W}$ and $z_\infty^i$'s are defined as above. Set $w_\infty^i = (\lambda\op{I}-(\op{A+D}_2\op{C}))^{-1}z_\infty^i$. As shown above in \eqref{Eq:xWc}, we have  $w_\infty^i\in\ssp{W}$. Since $w_i\in\ssp{W}_\mr{f}\subseteq\ssp{W}$ it follows that $(\lambda\op{I}-(\op{A+D}_2\op{C}))^{-1}z_c^i=w_i-w_\infty^i\in\ssp{W}$. Given that $\spanset{z_c^i}_{i=1}^{n_\mr{f}}$ is a basis of $\ssp{W}_c$, we obtain $(\lambda\op{I}-(\op{A+D}_2\op{C}))^{-1}\ssp{W}_c\subseteq\ssp{W}$.
	\end{enumerate}
		
	Finally,  for every $y\in\ssp{W}$ one can write $y = y_c+y_\infty$, where $y_c\in\ssp{W}_c$ and $y_\infty\in\ssp{W}_\infty$. As we have shown above $(\lambda\op{I}-(\op{A+D}_2\op{C}))^{-1}\ssp{W}_\infty\subseteq\ssp{W}$ and $(\lambda\op{I}-(\op{A+D}_2\op{C}))^{-1}\ssp{W}_c\subseteq\ssp{W}$. Therefore, $(\lambda\op{I}-(\op{A+D}_2\op{C}))^{-1}y\in\ssp{W}$, and consequently $(\lambda\op{I}-(\op{A+D}_2\op{C}))^{-1}\ssp{W}\subseteq\ssp{W}$. This completes the proof of the lemma.
\end{proof}
A bounded operator $\op{D}$ is called a \underline{friend} of the  $\fld{T}$-conditioned invariant subspace $\ssp{W}$ if $\SG{T}{A+DC}\ssp{W}\subseteq\ssp{W}$. The set of all friend operators of $\ssp{W}$ is denoted by $\FrOp{D}{W}$. Let $\op{D}\in\FrOp{D}{W}$ and consider a bounded operator $\op{D}_0$. As in Fin-D systems \cite{Mass_Thesis} (page 31), it follows (by using the above lemma) that a sufficient condition for $\op{D}_0$ to be a friend of $\ssp{W}$ is $(\op{D}-\op{D}_0)\op{C}\ssp{W}\subseteq\ssp{W}$.

We are now in a position to state the main results of this subsection leading us to the necessary and sufficient conditions for the $\fld{T}$-conditioned invariance of regular RS systems.
\begin{theorem}\label{Thm:CA_T_Inv}
	Consider the regular RS system \eqref{Eq:GeneralID} such that the operator $\op{C}$ is defined according to \eqref{Eq:OuputOperator}. The \CA-invariant subspace $\ssp{W}$ is an $\fld{T}$-conditioned invariant subspace if and only if
	\begin{equation}\label{Eq:TCA-InvCon}
	\ssp{W}=\ssp{W}_\phi+\ssp{W}_\mathrm{f},
	\end{equation}
	and $\overline{\dom{A}\cap\ssp{W}} = \ssp{W}$, where $\dim(\ssp{W}_\mr{f})<\infty$ and $\ssp{W}_\phi$ is the largest subspace contained in $\ssp{W}$ that can be expressed as
	\begin{equation}\label{Eq:EigeSpan}
	\ssp{W}_\phi = \overline{\spanset{\ssp{E}_i}_{i\in\fld{I}}}\;,
	\end{equation}
	in which $\ssp{E}_i$'s are the sub-eigenspaces of $\op{A}$ and $\fld{I}\subseteq\fld{N}$.
\end{theorem}
\begin{proof}
	({\bf If part}): Let $\ssp{W}=\ssp{W}_\phi+\ssp{W}_\mathrm{f}$. We show that $\ssp{W}$ can be spanned by the eigenspaces of $\op{A+DC}$, for a bounded $\op{D}$ (and therefore according to Corollary \ref{Col:A-T-Inv}, $\ssp{W}$ is $\SG{T}{A+DC}$-invariant). By invoking Lemma \ref{Lem:DiffOp4CondInv} we need to show this property for only one $\op{D}\in\underline{\op{D}}(\ssp{W})$.
	 Without loss of any generality, assume that  $\ssp{W}_\phi\cap\ssp{W}_\mr{f}=0$ (if $\ssp{W}_1 = \ssp{W}_\phi\cap\ssp{W}_\mr{f}\neq0$, redefine $\ssp{W}_\mr{f}$ to $\ssp{W}_\mr{f} = \ssp{W}_\mr{f}/\ssp{W}_1$).
	
	 First, we show that one can assume $\ssp{W}_\mr{f}\subset\dom{A}$ without loss of any generality. Since $\ssp{W}_\phi$ is $\SG{T}{A}$-invariant, it follows that  $\overline{\ssp{W}_\phi\cap\dom{A}}=\ssp{W}_\phi$ \cite{Zwart_Book}. Also, one can assume that $\overline{\ssp{W}\cap\dom{A}}=\ssp{W}$. If $\ssp{W}_\phi$ is Fin-D, $\ssp{W}$ is Fin-D, and hence $\ssp{W}_\mr{f}\subseteq\ssp{W}\subset\dom{A}$. Let, $\ssp{W}_\phi$ be Inf-D. By following along the same steps as in Lemma \ref{Lem:CA_equalency}, we define the basis $\{w_i\}_{i=1}^\infty$ of $\ssp{W}$ such that $w_i\in\dom{A}$ for all $i\in\fld{N}$ and $\{w_i\}_{i=n_\mr{f}+1}^\infty$ is a basis for $\ssp{W}_\phi$, where $n_\mr{f}=\dim(\ssp{W}_\mr{f})$ (since $\overline{\ssp{W}_\phi\cap\dom{A}}=\ssp{W}_\phi$ the existence of the basis  $\{w_i\}_{i=n_\mr{f}+1}^\infty$ is guaranteed). Let us set $\ssp{W}_\mr{ff} = \spanset{w_i}_{i=1}^{n_\mr{f}}\subset\dom{A}$, where it follows that $\ssp{W}=\ssp{W}_\phi+\ssp{W}_\mr{ff}$. Therefore, without loss of any generality, we assume $\ssp{W}_\mr{f}=\ssp{W}_\mr{ff}\subset\dom{A}$.
	
	Second, to show the result we first construct the bounded operator $\op{D}$ such that (i) $(\op{A+DC})(\ssp{W}\cap\dom{A})\subseteq\ssp{W}$, and (ii) $\op{DC}\ssp{W}_\phi=0$.  Define $\ssp{W}_\mr{fpc} = \ssp{W}_\mr{f}\cap(\ssp{W}_\mr{f}\cap\ker\op{C})^\perp$ and $\ssp{W}_\mr{fc}=\{w|w\in\ssp{W}_\mr{fpc}\; ,\; \op{C}w\neq\op{C}w_\phi, \forall w_\phi\in\ssp{W}_\phi\}$. In other words, $\ssp{W}_\mr{fc}$ is the largest subspace in $\ssp{W}_\mr{fpc}$ such that $\ssp{W}_\mr{fc}\cap\ker\op{C}=0$ and $\op{C}\ssp{W}_\mr{fc}\cap\op{C}\ssp{W}_\phi=0$. Moreover, by the definition of $\ssp{W}_\mr{fpc}$, we obtain $\ker\op{C}+\ssp{W}_\mr{f}/\ssp{W}_\mr{fc} =\ker\op{C}+ \ssp{W}_\mr{fpc}/\ssp{W}_\mr{fc}$. Since $\ssp{W}_\mr{f}\subset\dom{A}$, we have $\ssp{W}_\mr{fc}\subset\dom{A}$. Now, consider the operator $H_\mr{f}$ such that $\ker H_\mr{f}\op{C} = \ker\op{C}+\ssp{W}_\phi+\ssp{W}_\mr{f}/\ssp{W}_\mr{fc} =\ker\op{C}+\ssp{W}_\phi+\ssp{W}_\mr{fpc}/\ssp{W}_\mr{fc} $ and define $\op{C}_1=H_\mr{f}\op{C}$ (since $\ker\op{C}\subseteq\ker\op{C}_1$, there always exists a solution for $H_\mr{f}$). First, we show that $\ssp{W}$ is also an ($\op{C}_1$,$\op{A}$)-invariant subspace in two steps as follows.
	\begin{enumerate}
		\item  Let $w\in\ssp{W}_\mr{fpc}/\ssp{W}_\mr{fc}$. We show that $\op{A}w\in\ssp{W}$ (if $\ssp{W}_\mr{fpc}=\ssp{W}_\mr{fc}$, we have $w=0$ and we skip this step).  Since $\ssp{W}_\mr{fpc}\subset\ssp{W}_\mr{f}$, $\ssp{W}_\mr{fc}\subset\ssp{W}_\mr{f}$ and $\ssp{W}_\mr{f}\subset\dom{A}$, it follows that $w\in\dom{A}$. By the definition of $\ssp{W}_\mr{fc}$, there exists a $w_\phi\in\ssp{W}_\phi$ such that $\op{C}w=\op{C}w_\phi\neq0$. Next, we show that $w_\phi\in\dom{A}$. Let $\ssp{W}_\phi^\mr{p}\subset\ssp{W}_\phi$ be the subspace such that $\op{C}\ssp{W}_\phi^\mr{p} = \op{C}(\ssp{W}_\mr{fpc}/\ssp{W}_\mr{fc})$ and $\dim(\ssp{W}_\phi^\mr{p})=\dim(\ssp{W}_\mr{fpc}/\ssp{W}_\mr{fc})$. Also, let $\{w_\phi^i\}_{i=1}^\infty$ be a basis of $\ssp{W}_\phi$ such that $w_\phi^i\in\dom{A}$ (since $\overline{\ssp{W}_\phi\cap\dom{A}}=\ssp{W}_\phi$, this basis exists). By following along the same steps as in Lemma \ref{Lem:CA_equalency}, we can assume $w_\phi^i$'s such that  $w_\phi^i\in\ssp{W}_\phi^p$ for all $i\leq n_\phi$ and $w_\phi^i\in\ssp{W}_\phi/\ssp{W}_\phi^p$ for $i>n_\phi$. Therefore, since $\op{C}$ on $\ssp{W}_\mr{fpc}/\ssp{W}_\mr{fc}$ is bijective,  one can find $w_\phi\in\spanset{w_\phi^i}_{i=1}^{n_\phi}$ such that $\op{C}w = \op{C}w_\phi$, and since $w_\phi^i\in\dom{A}$, it follows that $w_\phi\in\dom{A}$. Now, let us set $w_c = (w-w_\phi)\in\ssp{W}\cap\ker\op{C}\cap\dom{A}$. Since $\op{A}w_\phi\in\ssp{W}$ (recall $\ssp{W}_\phi$ is $\op{A}$-invariant), and $\op{A}(\ssp{W}\cap\ker\op{C}\cap\dom{A})\subseteq\ssp{W}$, it follows that $\op{A}w\in\ssp{W}$.
		\item By considering the subspace $\ssp{W}_\phi^\mr{p}$, we decompose $\ssp{W}_\phi$ as $\ssp{W}_\phi = \ssp{W}_\phi^\mr{p}+\ssp{W}_\phi^\mr{c}+\ssp{W}_\phi\cap\ker\op{C}$, where $\ssp{W}_\phi^\mr{c}\cap\ssp{W}_\phi^\mr{p}=0$ and $\ssp{W}_\phi^\mr{c}\cap\ker\op{C}=0$. Similar to the above analysis we can assume $\ssp{W}_\phi^c\subset\dom{A}$ (i.e., there exists a subspace $\ssp{W}_\phi^c\subset\dom{A}$ that satisfies the above conditions).  By the definition of $H_\mr{f}$, it follows that $\ker H_\mr{f}\op{C} = \ker\op{C}+(\ssp{W}_\mr{fpc}/\ssp{W}_\mr{fc})+\ssp{W}_\phi^\mr{p}+\ssp{W}_\phi^\mr{c}$. Let $w\in(\ssp{W}\cap\ker\op{C}_1\cap\dom{A})$. It follows that $w=w_\mr{p}+w_\phi+w_\infty$, where $w_\mr{p}\in\ssp{W}_\mr{fpc}/\ssp{W}_\mr{fc}\subset\dom{A}$, $w_\phi\in(\ssp{W}_\phi^\mr{p}+\ssp{W}_\phi^\mr{c})\subset\dom{A}$ and $w_\infty\in\ssp{W}\cap\ker\op{C}$. Since $w,w_\mr{p},w_\phi\in\dom{A}$, it follows that $w_\infty\in\dom{A}$. As shown above, $\op{A}w_{p}\in\ssp{W}$, $\op{A}w_\phi\in\ssp{W}_\phi\subseteq\ssp{W}$ (since $\ssp{W}_\phi$ is $\op{A}$-invariant) and also $\op{A}w_\infty\in\ssp{W}$ (recall that $\ssp{W}$ is \CA-invariant). Therefore, $\op{A}w\in\ssp{W}$, and consequently $\op{A}(\ssp{W}\cap\ker\op{C}_1\cap\dom{A})\subseteq\ssp{W}$.
	\end{enumerate}
	Third, by following along the same steps as in Lemma \ref{Lem:CA_equalency}, we construct $\op{D}_\mr{f}$ such that $(\op{A+D}_\mr{f}\op{C}_1)(\ssp{W}\cap\dom{A})\subseteq\ssp{W}$. By setting $\op{D}= \op{D}_\mr{f}H_\mr{f}$, one can write $(\op{A+DC})(\ssp{W}\cap\dom{A})\subseteq\ssp{W}$.
	
	 Fourth, it should be pointed out that since $\ssp{W}_\phi\subseteq\ker H_\mr{f}\op{C}$ (refer to the definition of $H_\mr{f}$), we obtain  $\ssp{W}_\phi\subseteq\ker\op{C}_1$, and therefore, we have $\op{DC}\ssp{W}_\phi=\op{D}_\mr{f}\op{C}_1\ssp{W}_\phi=0$ . Consequently, it follows that every sub-eigenspace $\ssp{E}_i\subset\ssp{W}_\phi$ is also the sub-eigenspace of the operator $\op{A+DC}$. Therefore, $(\lambda\op{I}-(\op{A+DC}))^{-1}\ssp{W}_\phi\subseteq\ssp{W}_\phi$. Moreover, recall that $\ssp{W}_\mr{f}\subset\dom{A}$ and  the operator $\op{D}_\mr{f}$ is also defined such that $(\op{A+D}_\mr{f}\op{C}_1)\ssp{W}_\mr{f}\subseteq\ssp{W}_\mr{f}$ (refer to the proof of Lemma \ref{Lem:CA_equalency}). Therefore, by invoking  Lemmas \ref{Lem:A-1_T-Inv} and \ref{Lem:CAInv}, we obtain $(\lambda\op{I}-(\op{A+D}_\mr{f}\op{C}_1))^{-1}\ssp{W}_\mr{f}\subseteq\ssp{W}_\mr{f}$, and consequently $(\lambda\op{I}-(\op{A+DC}))^{-1}\ssp{W}_\mr{f}\subseteq\ssp{W}_\mr{f}$.
	
	 Finally, by invoking Lemma \ref{Lem:A-1_T-Inv} and Corollary \ref{Col:A-T-Inv}, it follows that $\ssp{W}_\mr{f}$ is also a sum of sub-eigenspaces of $(\op{A+DC})$. Therefore, $\ssp{W}$ is spanned by the sub-eigenspaces of $(\op{A+DC})$, and again by invoking Corollary \ref{Col:A-T-Inv}, $\ssp{W}$ is $\fld{T}_{\op{A+DC}}$-invariant, that is $\fld{T}$-conditioned invariant.\\
	({\bf Only if part}): Consider $\ssp{W}$ to be $\fld{T}$-conditioned invariant. By Definition \ref{Def:ConditionInv}, item 3), there exists a bounded operator $\op{D}$ such that $\ssp{W}$ is $\fld{T}_{\op{A+DC}}$-invariant (and also $(\op{A+DC})$-invariant) and $\ssp{W}=\overline{\spanset{\ssp{E}_i^D}_{i\in\fld{I}_D}}$, where $\ssp{E}_i^D$'s are the sub-eigenspaces of $(\op{A+DC})$. As in the first part of the proof, first we construct a bounded operator $\op{D}$ such that (i) $(\op{A+DC})(\ssp{W}\cap\dom{A})\subseteq\ssp{W}$, and (ii) $\op{DC}\ssp{W}_\phi=0$, where $\ssp{W}_\phi$ is the largest $\SG{T}{A}$-invariant contained in $\ssp{W}$. Consequently, we have $\ssp{W}=\overline{\ssp{W}_\phi+\ssp{W}_\mr{f}}$, and we then show that $\ssp{W}_\mr{f}$ is Fin-D.
	
	Let $\op{D}$ be  a bounded operator such that $\op{D} = \op{D}_\mr{f}H_\mr{f}$,  where $\ssp{W}_\phi$ is the largest $\SG{T}{A}$-invariant contained in $\ssp{W}$  (as expressed in equation \eqref{Eq:EigeSpan}) and $\ker H_\mr{f}\op{C} = \ssp{W}_\phi$. Moreover, $\op{D}_\mr{f}$ is defined by following along the same lines as in the proof of Lemma \ref{Lem:CA_equalency}.
	By using the fact that $\op{DC}\ssp{W}_\phi = 0$, it follows that $\ssp{W}_\phi = \overline{\spanset{\ssp{E}_j}_{j\in\fld{I}}}$, where $\fld{I}$ denotes an index set such that for each $j\in\fld{I}$ there exists an $i\in\fld{I}_D$ (recall $\ssp{W} = \overline{\spanset{\ssp{E}_i^D}_{i\in\fld{I}_D}}$ ) such that $\ssp{E}_j=\ssp{E}_i^D\subseteq(\ssp{W}\cap\ker H_\mr{f}\op{C})$.
 	
	Let us now set $\ssp{W} = \overline{\ssp{W}_\phi+\ssp{W}_\mr{f}}$, where $\ssp{W}_\mr{f}\cap\ssp{W}_\phi=0$. We show that $\dim(\ssp{W}_\mr{f})<\infty$ by contradiction. Since $\ssp{W}$ and $\ssp{W}_\phi$ are sums of sub-eigenspaces of $(\op{A+DC})$, it follows that $\ssp{W}_\mr{f}$ enjoys the same property. Let us assume that $\dim(\ssp{W}_\mr{f})=\infty$, and consider the subspace $\ssp{W}_\mr{fc}\subset\ssp{W}_\mr{f}$ such that  $\ssp{W}_\mr{fc}\subset\dom{A}$, $\ssp{W}_\mr{fc}\cap\ker H_\mr{f}\op{C}=0$ and $\ssp{W}_\mr{f} = \ssp{W}_\mr{fc}+\ssp{W}_\mr{f}\cap\ker H_\mr{f}\op{C}$ (following the above analysis since $H_\mr{f}\op{C}$ is  finite rank, by invoking the same steps as in the proof of Lemma \ref{Lem:CA_equalency},  the existence of this subspace can be guaranteed). Since $\ssp{W}_\mr{fc}\subset\dom{A}$ and $(\op{A+DC})\ssp{W}_\mr{fc}=0\subset\ssp{W}_\mr{fc}$ (refer to Lemma \ref{Lem:CA_equalency}, where we define the injection output operator), by invoking Lemma \ref{Lem:CAInv} and Corollary \ref{Col:A-T-Inv}, it follows that one can assume that $\ssp{W}_\mr{fc}$ is a sub-eigenspaces of $(\op{A+DC})$. Since $\ssp{W}_\mr{f}$ is a sum of sub-eigenspaces of $\op{A+DC}$, we obtain $\ssp{W}_\text{f}\cap\ker H_\mr{f}\op{C} = \overline{\spanset{\ssp{E}_i^D}_{i\in\fld{I}_\mr{f}}}+\ssp{W}_\mr{ff}$, where $\fld{I}_\mr{f}\subseteq\fld{I}_D$, and $\ssp{W}_\mr{ff}+\ssp{W}_\mr{fc}$ is also a sub-eigenspace of $(\op{A+DC})$ (note that it is possible to have  $\ssp{W}_\mr{ff}=0$). Since $\op{A+DC}$ is a regular RS operator (refer to Remarks \ref{Rem:CondOnDefRegRSOp} and \ref{Rem:SubEig_SimpleEig}), it is necessary to have $\dim(\ssp{W}_\mr{ff})<\infty$. Hence, since $\ssp{W}_\mr{f}$ is Inf-D, we obtain $\fld{I}_\mr{f}\neq\emptyset$. However, this is in contradiction with the definition of $\ssp{W}_\phi$ (that is the largest subspace in the form \eqref{Eq:EigeSpan}), and consequently $\ssp{W}_\mr{f}$ is a Fin-D subspace, and $\ssp{W} = \ssp{W}_\phi+\ssp{W}_\mr{f}$ (refer to Lemma \ref{Lem:SumClosedSpace}). This comp
letes the proof of the theorem.
\end{proof}
\begin{remark}\label{Rem:TCon_A-Inv_Fin}
	Theorem \ref{Thm:CA_T_Inv} shows that every $\fld{T}$-conditioned invariant subspace is constructed from a sum of the subspace $\ssp{W}_\phi$, that is $\SG{T}{A}$-invariant (and possibly Inf-D), and the Fin-D subspace $\ssp{W}_\mr{f}$ such that $\ssp{W}_\mr{f}\subseteq\dom{A}$ and $\ssp{W}_\mr{f}\cap\ssp{W}_\phi=0$. Given that $\ssp{W}$ is \CA-invariant and $\ssp{W}_\phi$ is $\op{A}$ invariant, it follows that  $\ssp{W}_\mr{f}$ is \CA-invariant. Hence, by invoking Lemma \ref{Lem:CAInv}, it follows that $\ssp{W}_\mr{f}$ is $\fld{T}$-conditioned invariant.
\end{remark}
For design of our subsequent FDI scheme, we need to obtain the smallest $\fld{T}$-conditioned invariant subspace (in the inclusion sense) containing a given subspace. The following lemma allows one to show that this smallest subspace always exists.
\begin{lemma}\label{Lem:TCA-Inve_CloseIntersect}
	The set of $\fld{T}$-conditioned invariant subspaces containing a given Fin-D subspace $\ssp{L}$ and satisfying the conditions of Theorem \ref{Thm:CA_T_Inv} is closed with respect to the intersection operator.
\end{lemma}
\begin{proof}
	Consider $\fld{T}$-conditioned invariant subspaces $\ssp{W}_1$ and $\ssp{W}_2$ containing $\ssp{L}$. Hence, $\op{A}(\ssp{W}_1\cap\ker \op{C}\cap\dom{A})\subseteq \ssp{W}_1$ and $\op{A}(\ssp{W}_2\cap\ker \op{C}\cap\dom{A})\subseteq \ssp{W}_2$, and consequently $\op{A}(\ssp{W}_1\cap \ssp{W}_2\cap\ker \op{C}\cap\dom{A})\subseteq \ssp{W}_1\cap\ssp{W}_2$. Also, given that $\ssp{W}_1$ and $\ssp{W}_2$ are closed, so does the subspace $\ssp{W}_1\cap\ssp{W}_2$. Therefore, $\ssp{W}_1\cap\ssp{W}_2$ is \CA-invariant. Moreover, $\ssp{W}_1\cap\ssp{W}_2\cap D(\op{A})$ is dense in $\ssp{W}_1\cap\ssp{W}_2$. Consequently, $\ssp{W}_1\cap\ssp{W}_2$ is feedback \CA-invariant (refer to Lemma \ref{Lem:CA_equalency}).
	
	By invoking Theorem \ref{Thm:CA_T_Inv}, let $\ssp{W}_{1}=\ssp{W}_{\phi_{1}}+\ssp{W}_{\mathrm{f_1}}, ~\ssp{W}_{2}=\ssp{W}_{\phi_{2}}+\ssp{W}_{\mathrm{f_2}}$ with $\ssp{W}_{\phi_{k}} = \overline{\spanset{\ssp{E}_i}_{i\in\fld{I}_{k}}},~k=1,2$, where we have $\ssp{W}_k= \eigssp{\fld{I}_k}+\ssp{W}_{\mr{f}_k}$, for $k=1,2$ ($\ssp{W}_{\mr{f}_k}\subset\dom{A}$  denotes two Fin-D subspaces - refer to Remark \ref{Rem:TCon_A-Inv_Fin}). Now, we show that $\ssp{W}_1\cap\ssp{W}_2$ can be represented  by $\eigssp{\fld{I}_3}+\ssp{W}_{\mr{f}_3}$. Let $x\in\eigssp{\fld{I}_1}\cap\eigssp{\fld{I}_2}$. Therefore, $x$ can be expressed as
	\begin{equation}
	\begin{split}
	x &= \sum_{i} \zeta_i^1\phi_i^1
	= \sum_{i} \zeta_i^2\phi_i^2,
	\end{split}
	\end{equation}
	where $\phi_i^1$ and $\phi_i^2$ denote the generalized eigenvectors  that span the subspaces $\eigssp{\fld{I}_1}$ and $\eigssp{\fld{I}_2}$, respectively. Since $\op{A}$ is a regular RS operator (i.e., only finitely many eigenvalues are repeated), therefore all but finitely many of the eigenspaces and the corresponding sub-eigenspace are equivalent. In other words, there are finitely many (generalized) eigenvectors corresponding to the same eigenvalue, and there are infinite eigenvectors for distinct eigenvalues (refer to Remark \ref{Rem:SubEig_SimpleEig}). By invoking Lemma \ref{Lem:RSUniqueRepresentation} (i.e., a unique representation of $x$), the fact that the (generalized) eigenvectors are independent,  it follows that $\eigssp{\fld{I}_1}\cap\eigssp{\fld{I}_2} = \eigssp{\fld{I}_3}+\ssp{W}_\mr{f_3}$, where $\ssp{W}_\mr{f_3}\subset\dom{A}$ (since $\ssp{E}_i\subset\dom{A}$) is a Fin-D subspace. Finally, given that $\ssp{W}_{\mr{f}_1}\subset\dom{A}$ and $\ssp{W}_{\mr{f}_2}\subset\dom{A}$ are Fin-D subspaces, it can be shown that $\ssp{W}_1\cap\ssp{W}_2 = \eigssp{\fld{I}_3}+\ssp{W}_{\mr{f}_4}$, where $\ssp{W}_{f_4}\subset\dom{A}$ is a Fin-D subspace. Hence, by invoking Theorem \ref{Thm:CA_T_Inv}, it follows that $\ssp{W}_1\cap\ssp{W}_2$ is a $\fld{T}$-conditioned invariant subspace. This completes the proof of the lemma.
\end{proof}
As shown in \cite{Curtain_Invariance_1986}, the smallest $\fld{T}$-conditioned invariant subspace containing  $\ssp{L}$ may not exist for a general Inf-D operator $\op{A}$. However, the fact that all but only finitely many eigenvalues of $\op{A}$ are simple plays a crucial role in the above proof to ensure that $\ssp{W}_{\mr{f}_3}\subset\dom{A}$.

We are now in a position to introduce our proposed algorithm for computing the smallest $\fld{T}$-conditioned invariant subspace containing a given subspace. The algorithm for computing the smallest \CA-invariant subspace containing a given subspace $\ssp{L}$ is given by \cite{Curtain_Invariance_1986}, namely
\begin{equation}\label{Eq:CAInvAlg_Oreginal}
\begin{split}
\ssp{W}^0 &=\ssp{L} ,\;\;
\ssp{W}^k =\overline{\ssp{L}+\op{A}(\ssp{W}^{k-1}\cap\ker\op{C}\cap\dom{A})}.
\end{split}
\end{equation}
As pointed out in \cite{Curtain_Invariance_1986}, the limit of the above algorithm may be a non-closed subspace, and consequently, it is {not conditioned invariant} in the sense of Definition \ref{Def:ConditionInv}.
Below, we now provide an algorithm that computes the minimum $\fld{T}$-conditioned invariant subspace in a \emph{finite}  number of steps provided that the subspace $\ssp{N}_\op{A} = \bigcap_{n\in\underline{\fld{N}}} \ker\op{C}\op{A}^n$, which denotes the $\op{A}$-unobservable subspace of the system \eqref{Eq:GeneralID}, is known.
\begin{theorem}\label{Thm:CAInvAlg}
	Consider the RS system \eqref{Eq:GeneralID} and a given Fin-D subspace $\ssp{L}\subset\dom{A}$ and $\ssp{L}\cap\ker\op{C}\subset D(\op{A}^\infty)$, where $D(\op{A}^\infty)=\bigcap_{k=1}^\infty D(\op{A}^k)$  that is decomposed into disjoint subspaces $\ssp{L}= \ssp{L}_{\ssp{N}^\perp} + \ssp{L}_\ssp{N}$, such that $\ssp{L}_{\ssp{N}^\perp}\cap\ssp{N}_\op{A}=0$ and $\ssp{L}_\ssp{N}=\ssp{L}\cap\ssp{N}_\op{A}$. The smallest $\fld{T}$-conditioned invariant subspace containing $\ssp{L}$ (as denoted by $\ssp{W}^*$) is given by $\ssp{W}^* = \ssp{W}_\ell+ \ssp{Z}^*$, where $\ssp{Z}^*$ is the limiting subspace of the following algorithm
	\begin{equation}\label{Eq:T-CondSpaceAlg}
	\begin{split}
	\ssp{Z}_0 &=\ssp{L}_{\ssp{N}^\bot},\;\;
	\ssp{Z}_k = \ssp{L}_{\ssp{N}^\bot}+ \op{A}(\ssp{Z}_{k-1}\cap\ker\op{C}\cap\dom{A}),\\
	\end{split}
	\end{equation}
	and $\ssp{W}_\ell=\eigssp{\fld{J}}$ denotes the smallest subspace in the form of  \eqref{Eq:EigeSpan} (sum of the sub-eigenspaces of $\op{A}$) such that $\ssp{L}_\ssp{N}\subseteq\ssp{W}_\ell$. Moreover, the above algorithm converges in a finite number of steps.
\end{theorem}
\begin{proof}
	First, we show that this algorithm converges in a finite number of steps by contradiction. Assume that there exists at least a vector $x\in\ssp{L}_{\ssp{N}^\bot}\cap D(\op{A}^\infty)$ such that $\op{A}^nx\subseteq\ker\op{C}$ and $\op{A}^nx$ are independent vectors for all $n$. Otherwise, there is an $n_0$ such that $\op{A}^{n_0}x\notin\ker\op{C}$ for all  $x\in\ssp{L}_{\ssp{N}^\perp}$. Therefore, $(\ssp{Z}_{n_0+1}\cap\ker\op{C}\cap\dom{A}) =(\ssp{Z}_{n_0}\cap\ker\op{C}\cap\dom{A})$, and consequently we obtain $\ssp{Z}_{n_0+2} = \ssp{Z}_{n_0+1}$. Consequently, the above algorithm converges in a finite number of steps. Since $\ker\op{C}$ is a closed subspace, we have $\op{A}^nx\in\ker\op{C}$ for all $n\in\fld{N}$ and $\lim_{n\rightarrow\infty}\op{A}^nx\in\ker\op{C}$ (if $\lim_{n\rightarrow\infty}\op{A}^nx$ exists), and consequently $x\in\ssp{N}_\op{A}$, which is in contradiction with the fact that $\ssp{L}_{\ssp{N}^\bot}\cap\ssp{N}_\op{A} = 0$. Therefore, there exists a $k\in\fld{N}$ such that $\ssp{Z}^* = \ssp{Z}_k$. Moreover, since $\ssp{L}\cap\ker\op{C}\subset D(\op{A}^\infty)$, it follows that $\ssp{Z}^*\subset\dom{A}$.
	
	Second,  since $\ssp{L}$ is  Fin-D it follows that $\dim(\ssp{Z}^*)<\infty$. By considering the definition of $\ssp{W}_\ell$, we obtain $\overline{\ssp{W}^*\cap\dom{A}}=\ssp{W}^*$, and by invoking Theorem \ref{Thm:CA_T_Inv}, it follows that $\ssp{W}^*$ is a $\fld{T}$-conditioned invariant subspace.
	
	Finally,  we show that $\ssp{W}^*$ is the smallest $\fld{T}$-conditioned invariant subspace. Consider a $\fld{T}$-conditioned invariant subspace $\ssp{W}$ such that $\ssp{L}\subseteq\ssp{W}$.
	Given that $\ssp{W}$ is $\fld{T}$-conditioned invariant and $\op{A+DC}$ is a regular RS operator (refer to Remark \ref{Rem:CondOnDefRegRSOp}), $\ssp{W} = \overline{\spanset{\ssp{E}_i^D}_{i\in\fld{I}}}$, where $\fld{I}\subseteq\fld{N}$ and $\ssp{E}_i^D$ is a sub-eigenspace of $\op{A+DC}$.
	Next, we  show that $(\ssp{W}_\ell+\ssp{L})\subseteq\ssp{W}$. Towards this end,  let $\op{D}$ be the injection operator that is defined as in the proof of Theorem \ref{Thm:CA_T_Inv}, where $\ssp{W}=\ssp{W}_\phi+\ssp{W}_\mr{f}$ and $\op{DC}\ssp{W}_\phi=0$. Also, following along the above one can assume that there is no sub-eigenspace $\ssp{E}$ of $\op{A}$ such that $\ssp{E}\subset\ssp{W}_\mr{f}$ (i.e., $\ssp{W}_\phi$ is the largest subspace in the form \eqref{Eq:EigeSpan} that is contained in $\ssp{W}$).  Since  $\ssp{L}_\ssp{N}\subseteq\ssp{N}_\op{A}$, and consequently $\ssp{L}_\ssp{N}\subset D(\op{A}^\infty)$, it follows that $(\lambda\op{I-A})^k\ssp{L}_\ssp{N}=(\lambda\op{I-(A+DC)})^k\ssp{L}_\ssp{N}\subset\ker\op{C}$ for all $k\in\fld{N}$. Therefore, $\ssp{L}_\ssp{N}\subseteq\ssp{W}_\phi$. Otherwise, if $\ssp{L}_\ssp{N}\cap\ssp{W}_\mr{f}\neq0$, there exists an $x\in\ssp{L}_\ssp{N}\cap\ssp{W}_\mr{f}$ such that $(\lambda\op{I-A})^kx\in\ker\op{C}\cap\ssp{W}_\mr{f}$ for all $k\in\fld{N}$ (recall that $\ssp{W}_\mr{f}$ is \CA-invariant). Since, $\ssp{W}_\mr{f}$ is Fin-D, it follows that there exists a sub-eigenspace that is contained in $\ssp{W}_\mr{f}$, and this is in contradiction with the definition of $\ssp{W}_\mr{f}$. Since $\ssp{W}_\ell$ is the smallest subspace in the form of  \eqref{Eq:EigeSpan} such that $\ssp{L}_\ssp{N}\subseteq\ssp{W}_\ell$, it follows that $\ssp{W}_\ell\subset\ssp{W}_\phi$. Furthermore, given that we assume $\ssp{L}\subseteq\ssp{W}$,  we obtain $(\ssp{W}_\ell+\ssp{L})\subseteq\ssp{W}$. Now, since the algorithm is increasing and starts from $\ssp{L}_{\ssp{N}^\perp}\subseteq\ssp{L}\subseteq\ssp{W}$, we obtain $\ssp{Z}_k\subseteq\ssp{W}$, and consequently $\ssp{W}^*\subseteq\ssp{W}$. It follows that $\ssp{W}^*$ is the smallest $\fld{T}$-conditioned invariant subspace containing $\ssp{L}$. This completes the proof of the lemma.
\end{proof}
It should be pointed out that one can compute $\ssp{W}_\ell$ as follows.
\begin{enumerate}
	\item Let $\op{X}_\text{inf}=\overline{\spanset{\ssp{E}_i}_{i\in\fld{J}_s}}$ and $\op{X}_\text{f} = \spanset{\ssp{E}_j}_{j\in\fld{J}_m}$, where $\fld{J}_s$ and $\fld{J}_m$ denote the index sets for simple and multiple (or repeated) eigenvalues, respectively. Also, $\ssp{E}_i$'s and  $\ssp{E}_j$'s denote the sub-eigenspaces that correspond to the simple and multiple (or repeated) eigenvalues, respectively (note that $\dim(\op{X}_\text{f})<\infty$).
	\item Compute, $\ssp{W}_\ell^m$, the smallest sub-eigenspace in $\op{X}_\text{f}$ containing $\op{P}_{\text{f}}\ssp{L}_\ssp{N}$, where $\op{P}_\text{f}$ denotes the projection from $\op{X}$ onto $\op{X}_\text{f}$. It follows that $\ssp{W}_\ell^m=\spanset{\phi_k}_{k\in\fld{I}_m}$, where $\fld{I}_m\subseteq\fld{J}_m$,  and therefore $\dim(\ssp{W}_\ell^m)<\infty$.
	\item Let $\ssp{W}_\ell^s = \overline{\spanset{\ssp{E}_k}_{k\in\fld{I}_s}}$, where $\fld{I}_s\subseteq\fld{J}_s$ and the eigenvector $\phi_k\in\ssp{E}_k$ (that corresponds to $\lambda_k$) does appear in the representation of at least one  member of $\ssp{L}_{\ssp{N}}$ (refer to Lemma \ref{Lem:RSUniqueRepresentation}).
	\item Set $\ssp{W}_\ell = \ssp{W}_\ell^s+\ssp{W}_\ell^m$.
\end{enumerate}

\vspace{-2mm}
\subsection{Unobservability Subspace}\label{Sec:UnObse}
In the geometric FDI approach, one needs to work with another invariant subspace known as the unobservability subspace. In this subsection, we first provide two definitions for this subspace, and then develop an algorithm to construct it computationally. 

\begin{definition}\label{Def:UnobservabilitySpace}\
	\begin{enumerate}
		\item The subspace $\op{S}$ is called an $\op{A}$-unobservability subspace for the RS system \eqref{Eq:GeneralID}, if there exist two bounded operators $\op{D}:\fld{R}^q\rightarrow \op{X}$ and $H:\fld{R}^q\rightarrow \fld{R}^{q_h}$, where $q_h\leq q$, such that $\op{S}$ is the largest $\op{A+DC}$-invariant subspace contained in $\ker H\op{C}$ (i.e., $\op{S}=<\ker H\op{C}|\op{A+DC}>$).
		\item The subspace $\op{S}$ is called an unobservability subspace for the RS system \eqref{Eq:GeneralID}, if there exist two bounded operators $\op{D}:\fld{R}^q\rightarrow \op{X}$ and $H:\fld{R}^q\rightarrow \fld{R}^{q_h}$,  where $q_h\leq q$, such that $\op{S}$ is the largest $\SG{T}{A+DC}$-invariant subspace contained in $\ker H\op{C}$ (i.e., $\op{S}=<\ker HC|\SG{T}{A+DC}>$).
	\end{enumerate}
\end{definition}
\begin{remark}\label{Rem:Unobs}
	It follows that the $\op{A}$- and unobservability subspaces are the $\op{A}$- and unobservable  subspaces of the pair ($\op{HC}$,$\op{A+DC}$), respectively. Also, by definition $\op{A}$- and unobservability subspaces are also feedback \CA- and $\fld{T}$-conditioned invariant, respectively.
\end{remark}

\noindent
\underline{\textbf{The Unobservability Subspace Computing Algorithm:}}
As stated earlier, for the FDI problem one is interested in computing the smallest unobservability subspace containing a given subspace. 
By following along the same lines as in Lemma \ref{Lem:TCA-Inve_CloseIntersect}, and the fact that $\op{A+DC}$ is a regular operator, and finally by invoking Remark \ref{Rem:Unobs}, one can show that the set of all unobservablity subspaces containing a given subspace always admits a minimum in the inclusion sense. In the Fin-D case, the unobservability subspace computing algorithm involves the inverse image of certain subspaces with respect to the state dynamic operator (i.e., the operator $A$) \cite{Mass_Thesis} (equation 2.61). However, for Inf-D systems, it is not convenient to deal with the inverse image of $\op{A}$ (if $0\not\in\roi{A}$). To overcome this difficulty, one can compute the unobservability subspace by using its dual subspace which is the controllability subspace. Therefore, one needs to compute the adjoint operators of $\op{A}$ and $\op{C}$ as was pointed out in \cite{ECC2014}.

The method in \cite{ECC2014} uses a non-decreasing algorithm that converges in a countable number of steps. However, since the algorithm is non-decreasing, the limiting subspace is not necessarily closed. Another approach for computing the unobservability subspace would be to use the resolvent operator $(\lambda\op{I}-\op{A})^{-1}$. This approach is more feasible given that one deals with $\fld{T}$-conditioned invariant subspaces and with $(\lambda\op{I}-\op{A})^{-1}$, which is a bounded operator. Moreover, the corresponding algorithm will be non-increasing and converges in a countable number of steps. Consequently, this will ensure that the limiting subspace will be closed \cite{Curtain_Invariance_1986}.  The following theorem provides  an approach to compute the smallest unobservability subspace containing a given Fin-D subspace $\ssp{L}$.

\begin{theorem}\label{Thm:UnobsAlg-Countable}
	Consider the model \eqref{Eq:GeneralID} which is assumed to be a  regular RS system and a given Fin-D subspace $\ssp{L}\subset \dom{A}$. Let $\ssp{W}^*$ denote the smallest $\fld{T}$-conditioned invariant subspace containing  $\ssp{L}$, where $\ssp{W}^* = \ssp{W}_\phi^*+\ssp{W}_\mr{f}^*$ (from Theorem \ref{Thm:CA_T_Inv}), $\ssp{W}_\phi^*$ denote the subspace contained in $\ssp{W}^*$ in the form \eqref{Eq:EigeSpan} and $\ssp{W}_\mr{f}^*\subset\dom{A}$ denote a Fin-D subspace. The smallest unobservability subspace containing $\ssp{L}$ (denoted by $\op{S}^*$) is given by
	\begin{equation}\label{Eq:T-UnobSpaceAlg}
	\begin{split}
	\op{S}^* = \overline{\ssp{W}_\phi^*+\ssp{N}}+\ssp{W}_{\phi,\mr{f}}^*,
	\end{split}
	\end{equation}
	in which $\ssp{N}$ is the unobservable subspace of \CA, $\ssp{W}_{\phi,\mr{f}}^*$ is the largest subspace in the form of $\overline{\spanset{\ssp{E}_i^D}_{i\in\fld{I}_D}}$ such that $\ssp{W}_{\phi,\mr{f}}^*$ contains $\ssp{W}_\mr{f}^*$ and is contained in $\overline{\ssp{W}^*+\ker\op{C}}$. Also, $\ssp{E}_i^D$'s denote the sub-eigenspaces of  $(\op{A+DC})$.
\end{theorem}
\begin{proof}
	Let us first show that $\op{S}^*$ is a $\fld{T}$-conditioned invariant subspace. Since $\ssp{N}$ is $\SG{T}{A}$-invariant, we obtain $\ssp{N} = \eigssp{\fld{I}}$, where $\ssp{E}_i$'s denote the sub-eigenspaces of $\op{A}$ (by using Corollary \ref{Col:A-T-Inv}). Let $\op{D}\in\underline{D}(\ssp{W}^*)$ that is constructed as in Theorem \ref{Thm:CA_T_Inv} (i.e., $\op{DC}\ssp{W}_\phi=0$ and $(\lambda\op{I-(A+DC)})^{-1}\ssp{W}^*\subseteq\ssp{W}^*$). Since $\ssp{N}\subseteq\ker\op{C}$, as shown above (in the proof of Theorem \ref{Thm:CA_T_Inv}) $\ssp{E}_i$'s are also sub-eigenspaces of $(\op{A+DC})$. Also, by definition, $\ssp{W}_{\phi,\mr{f}}^*$ is a sum of sub-eigenspaces of $(\op{A+DC})$. Therefore, $\op{S}^*$ is a sum of sub-eigenspaces of $(\op{A+DC})$ and by invoking Corollary \ref{Col:A-T-Inv}, it follows that $\op{S}^*$ is $\SG{T}{A+DC}$-invariant (i.e., $\fld{T}$-conditioned invariant).
	
	Second, let $H$ denote a map such that $\ker H\op{C} = \overline{\ssp{W}^*+\ker\op{C}}$ (one choice is $H:\fld{R}^q\rightarrow\fld{R}^{q_h}$, where $\ker H=\ssp{W}^*\cap(\ssp{W}^*\cap\ker\op{C})^\perp$). Since $\ssp{W}_{\phi,\mr{f}}^*\subseteq\overline{\ssp{W}^*+\ker\op{C}}$, and $\ssp{W}_\mr{f}^*\subseteq\ssp{W}_{\phi,\mr{f}}^*$, it follows that $\overline{\ssp{W}_\phi^*+\ker\op{C}}+\ssp{W}_{\phi,\mr{f}}^* = \overline{\ssp{W}^*+\ker\op{C}}$. Also, given that $\ssp{N}\subseteq\ker\op{C}$, we obtain  $\overline{\ssp{W}^*+\ker\op{C}}=\overline{\op{S}^*+\ker\op{C}}$, and consequently, we have $\op{S}^*\subseteq \ker H\op{C}$.
	
	Third, we show that $\op{S}^*$ is an unobservable subspace of the system ($H\op{C}$, $\op{A+DC}$). As shown above $\op{S}^* = \overline{\spanset{\ssp{E}_i^D}_{i\in\fld{I}}}$, where $\ssp{E}_i^D$ is a sub-eigenspace of $\op{A+DC}$. Next, it is shown that $\op{S}^*$ contains all sub-eigenspaces of $(\op{A+DC})$ that are contained in $\ker H\op{C}$. Let $\ssp{E}_0^D$ denote a given sub-eigenspace of $\op{A+DC}$, such that $\ssp{E}_0^D\subseteq\ker H\op{C}$. If $\ssp{E}_0^D\not\subseteq\ker\op{C}$, since $\ssp{W}_\phi^*+\ssp{W}_{\phi,\mr{f}}^*$ contains all sub-eigenspaces that may not be contained in $\ker\op{C}$ (recall the definition of $H$ and $\ssp{W}_{\phi,\mr{f}}^*$) but is contained  in $\ker H\op{C}$, we obtain  $\ssp{E}_0^D\subseteq(\ssp{W}_\phi^*+\ssp{W}_{\phi,\mr{f}}^*)\subseteq\op{S}^*$. Now, assume that $\ssp{E}_0^D\subseteq\ker\op{C}$. It follows that $(\lambda\op{I}-(\op{A+DC}))^{-1}\ssp{E}_0^D = (\lambda \op{I}-\op{A})^{-1}\ssp{E}_0^D\subseteq\ker\op{C}$,  and consequently, $\ssp{E}_0^D\subseteq\ssp{N}\subseteq\op{S}^*$. Hence, $\op{S}^*$ is the largest subspace contained in $\ker H\op{C}$  that is spanned by the sub-eigenspace of $\op{A+DC}$ (i.e., every sub-eigenspace in $\ker H\op{C}$ is contained in $\op{S}^*$).  Therefore, $\op{S}^*$ is the unobservable subspace of the pair ($H\op{C}$,$\op{A+DC}$).
	
	Finally, we show that $\op{S}^*$ is the smallest unobservability subspace containing $\ssp{L}$. Let $\op{S}$ denote another unobservability subspace containing $\ssp{L}$. Since $\op{S}$ is $\fld{T}$-conditioned invariant containing $\ssp{L}$, it follows that  $\ssp{W}^*\subseteq\op{S}$ ($\ssp{W}^*$ is the smallest $\fld{T}$-conditioned invariant containing $\ssp{L}$). Now, let $H_1$ be selected such that $\ker H_1\op{C} = \overline{\op{S}+\ker\op{C}}$. Since $\op{S}^*\subseteq\overline{\ssp{W}^*+\ker\op{C}}$, it follows $\op{S}^*\subseteq\ker H_1\op{C}$. Also, given that $\op{S}$ is the largest $\fld{T}$-conditioned invariant in $\ker H_1\op{C}$, by invoking Theorem \ref{Thm:CA_T_Inv}, $\op{S}$ is the largest subspace in the form \eqref{Eq:TCA-InvCon} that is contained in $\ker H_1\op{C}$. Since $\op{S}^*$ is also expressed in the form \eqref{Eq:TCA-InvCon} (since $\op{S}^*$ is also $\fld{T}$-conditioned invariant), it follows that $\op{S}^*\subseteq\op{S}$. This completes the proof of the theorem.
\end{proof}

It should be pointed out that since $\ssp{W}_\mr{f}^*$ is Fin-D and the operator $\op{A+DC}$ is regular RS, $\ssp{W}_{\phi,\mr{f}}^*$ is Fin-D. Therefore, one can compute $\ssp{W}_{\phi,\mr{f}}^*$ based on the sub-eigenspaces of $\op{A+DC}$ (i.e., for every sub-eigenspace $\ssp{E}_0^D$ of $\op{A+DC}$ that (i) is contained in $\overline{\ssp{W}^*+\ker\op{C}}$, (ii) $\ssp{E}_0^D\not\subseteq\overline{\ssp{W}_\phi^*+\ssp{N}}$, and (iii) $\ssp{E}_0^D\not\perp\ssp{W}_\mr{f}^*$, we have $\ssp{E}_0^D\subseteq\ssp{W}_{\phi,\mr{f}}^*$).

\subsection{Controlled Invariant Subspaces and the Duality Property}\label{Sec:Dual}

As stated above, for addressing the FDI problem one needs to construct the conditioned invariant subspace. However, for the  disturbance decoupling problem the controlled invariant subspaces (that are dual to the conditioned invariant subspaces)  are needed. For sake of completeness of this paper, in this subsection we review controlled invariant subspaces of the RS system \eqref{Eq:GeneralID}, where necessary and sufficient conditions for  the controlled invariance are provided. We address the controlled invariant subspaces by using the duality property. Moreover, we compare our results with those that are currently available in the literature \cite{Pandolfi_Disturbance_86,SchmidtInv_1980,Byrnes1998}.

Similar to conditioned invariant subspaces, there are \emph{three types} of controlled invariant subspaces. These are discussed further below.
\begin{definition}\label{Def:ControlledInv}\cite{Curtain_Invariance_1986}
	Consider the closed subspace $\ssp{V}\subseteq\op{X}$ and $\ssp{B} = \ima \op{B}$, where $\op{B}$ is defined from the system \eqref{Eq:GeneralID}. Then,
	\begin{enumerate}
		\item $\ssp{V}$ is called \AB-invariant if $\op{A}(\ssp{V}\cap\dom{A})\subseteq\overline{\ssp{V}+\ssp{B}} = \ssp{V}+\ssp{B}$ (since $\dim(\ssp{B})<\infty$).
		\item $\ssp{V}$ is called feedback \AB-invariant if there exists a bounded operator $\op{F}:\op{X}\rightarrow\fld{R}^m$ such that $(\op{A+BF})(\ssp{V}\cap\dom{A})\subseteq\ssp{V}$.
		\item $\ssp{V}$ is called $\fld{T}$-controlled invariant if there exists a bounded operator $\op{F}:\op{X}\rightarrow\fld{R}^m$ such that (i) the operator $\op{A+BF}$ is the infinitesimal generator of a $C_0$-semigroup $\fld{T}_\op{A+BF}$; and (ii) $\ssp{V}$ is invariant with respect to $\fld{T}_\op{A+BF}$ as per Definition \ref{Def:AInv}, item 2).
	\end{enumerate}
\end{definition}
In the literature, $\fld{T}$-controlled invariance is also called closed feedback invariance \cite{Zwart_Book} and $\fld{T}(\op{A},\op{B})$-invariance \cite{Curtain_Invariance_1986}. Following the above discussion, it can be shown that  Definition \ref{Def:ControlledInv}, item 3) $\Rightarrow$ item 2) $\Rightarrow$ item 1) \cite{Curtain_Invariance_1986}. In this subsection, we are interested in developing and addressing  necessary and sufficient conditions for equivalence of the above definitions.
In \cite{Curtain_Invariance_1986}, the duality between the  Definitions \ref{Def:ConditionInv} and \ref{Def:ControlledInv} was shown by using the following lemmas (the superscript $*$ is used for adjoint operators).
\begin{lemma}\cite{Curtain_Invariance_1986} (Lemma 5.2)\label{Lem:SubspaceAdjoint}
	Consider the system \eqref{Eq:GeneralID}, where $\op{A}$ is an infinitesimal generator of the $C_0$ semigroup $\SG{T}{A}$ (more general than the regular RS operator) and the operator $\op{C}$ is bounded (but not necessarily finite rank), and two subspaces $\ssp{S}_1$ and $\ssp{S}_2$. We have
	\begin{enumerate}
		\item $(\ssp{S}_1+\ssp{S}_2)^\perp = \overline{\ssp{S}_1^\perp+\ssp{S}_2^\perp}$.
		\item $(\ker\op{C})^\perp = \overline{\ima \op{C}^*}$.
		\item If $\SG{T}{A}\ssp{S}_1\subseteq\ssp{S}_2$, then $\fld{T}_{\op{A}^*}\ssp{S}_2^\perp\subseteq\ssp{S}_1^\perp$.
		\item If $\op{A}(\ssp{S}_1\cap\dom{A})\subseteq\ssp{S}_2$, then $\op{A}^*(\ssp{S}_2^\perp\cap\dom{A^*})\subseteq(\ssp{S}_1\cap\dom{A})^\perp$.
	\end{enumerate}
\end{lemma}
By using Lemma \ref{Lem:SubspaceAdjoint}, item 3) the following result can be obtained.
\begin{lemma}\cite{Curtain_Invariance_1986}\label{Lem:DualCurtain}
	Consider the regular RS system \eqref{Eq:GeneralID}. The subspace $\ssp{V}$ is $\fld{T}$-controlled invariant  if and only if $\ssp{V}^\perp$ is $\fld{T}$-conditioned invariant with respect to ($\op{B}^*$,$\op{A}^*$).
\end{lemma}
 The following lemma now directly provides our proposed result.
\begin{lemma}\label{Lem:AB_Equivalency}
	Consider the regular RS system \eqref{Eq:GeneralID} and the closed subspace $\ssp{V}$ such that $\overline{\ssp{V}\cap\dom{A}} = \ssp{V}$. The feedback \AB-invariance property is equivalent to the \AB-invariance property.
\end{lemma}
\begin{proof}
	It is sufficient to show  that \AB-invariance $\Rightarrow$  feedback \AB-invariance. Let $\ssp{V}$ be \AB-invariant. Since $\dom{A}$ is dense in $\ssp{V}$, one can construct the basis $\{v_i\}_{i\in\fld{I}}$ (where $\fld{I}\subseteq\fld{N}$) such that $v_i\in\dom{A}$. Since $\op{B}$ is finite rank,  we have $\ssp{V} = \ssp{V}_{\text{inf}}+\ssp{V}_\text{f}$, such that $\op{A}(\ssp{V}_{\text{inf}}\cap\dom{A})\subseteq\ssp{V}$, $\ssp{V}_\mr{f}\subset\dom{A}$ and $\op{A}v_i$' are linearly independent for all $i=1,\cdots,n_f$, where without loss of any generality we assume that $\ssp{V}_\text{f} = \spanset{v_i}_{i=1}^{n_\text{f}}$ and $\op{A}\ssp{V}_\mr{f}\subseteq\ssp{B}$ (by following along the same steps as in Lemma \ref{Lem:CA_equalency}). Therefore, there exist $u_i$'s such that $\op{A}v_i = -\op{B}u_i$ for all $i=1,\cdots,n_\text{f}$. Let us now define $F$ such that $F[v_1,\;\cdots,\;v_{n_\text{f}}] = [u_1,\;\cdots,\;u_{n_\text{f}}]$ (note since $\ker [v_1,\;\cdots,\;v_{n_\text{f}}] = 0$,  $F$ always exists), and let $\op{F}$ denote the extension of $F$ to $\op{X}$. In other words, for all $x\in\op{X}$, we have $\op{F}x=Fx_v$, where $x=x_{v^\perp}+x_v$, $x_v\in\ssp{V}_\text{f}$ and $x_{v^\perp}\perp\ssp{V}_\text{f}$. It follows that $||\op{F}||=||F||<\infty$ (i.e., $\op{F}$ is bounded) and $(\op{A+BF})(\ssp{V}\cap\dom{A})\subseteq\ssp{V}$. Therefore, $\ssp{V}$ is feedback \AB-invariant. This completes the proof of the lemma.	
\end{proof}
\begin{remark}\label{Rem:A-BoundedFeadback}
		The operator $\op{F}:\op{X}\rightarrow\op{Y}$ is $\op{A}$-bounded if $\dom{A}\subseteq\dom{F}$ and $\op{F}(\lambda \op{I-A})^{-1}$ is bounded (\cite{Zwart_Book}-Definition II.4). In \cite{Zwart_Book} feedback \AB-invariant is defined as follows. The subspace $\ssp{V}$ is feedback \AB-invariant if there exists an \underline{$\op{A}$-bounded} state feedback (as opposed to bounded state feedback as in Definition \eqref{Def:AInv}) $\op{F}$, such that $(\op{A+BF})(\ssp{V}\cap\dom{A})\subseteq\ssp{V}$. By this definition, in \cite{Zwart_Book} (Theorem II.26), it is shown that \AB-invariant and feedback \AB-invariant are equivalent.    However, Lemma \ref{Lem:AB_Equivalency} above achieves the same result (but by including an extra condition that is $\overline{\ssp{V}\cap\dom{A}} = \ssp{V}$) when we restrict the feedback to  bounded operators (i.e., as per Definition \ref{Def:AInv}). Note that this result cannot be concluded from Lemma II.25 and Theorem II.26 in \cite{Zwart_Book}.
\end{remark}

However, we are interested in deriving a direct  necessary and sufficient condition for the $\fld{T}$-controlled invariance property. By taking advantage of the duality property, the following theorem now provides the necessary and sufficient conditions for the $\fld{T}$-controlled invariance property.
\begin{theorem}\label{Thm:MyDual}
	Consider the regular RS system \eqref{Eq:GeneralID} and the closed subspace $\ssp{V}$ such that $\overline{\ssp{V}\cap\dom{A}}=\ssp{V}$ and $\op{A}(\ssp{V}\cap\dom{A})\subseteq\ssp{V}+\ima \op{B}$. Then, $\ssp{V}$ is $\fld{T}$-controlled invariant if and only if $\ssp{V}$ can be represented as $\ssp{V}=\ssp{V}_\phi\cap\ssp{V}_\mathrm{f}^\perp$, where $\ssp{V}_\mathrm{f}\subset\dom{A^*}$ is a Fin-D subspace and $\ssp{V}_\phi$ is the smallest subspace containing $\ssp{V}$ that can be expressed as
	\begin{equation}\label{Eq:EigeSpanCtrCond}
	\ssp{V}_\phi = \overline{\spanset{\ssp{E}_i}_{i\in\fld{I}}},
	\end{equation}
	in which $\ssp{E}_i$'s denote the sub-eigenspaces of $\op{A}$ and $\fld{I}\subseteq\fld{N}$.
\end{theorem}
\begin{proof}
	({\bf If part}): Let $\ssp{V}=\ssp{V}_\phi\cap\ssp{V}_\text{f}^\perp$. It follows that $\ssp{W}_\psi = \ssp{V}_\phi^\perp$ can be expressed as $\ssp{W}_\psi = \overline{\spanset{\ssp{E}_i^*}_{i\in\fld{I}}}$, where $\ssp{E}_i^*$'s denote sub-eigenspaces of $\op{A}^*$ (since $\ssp{W}_\psi$ is $\fld{T}_{\op{A}^*}$-invariant). Given that $\ssp{V}_\text{f}\subseteq\dom{A^*}$, $\dim(\ssp{V}_\text{f})<\infty$ and $\overline{\ssp{W}_\psi\cap\dom{A^*}}=\ssp{W}_\psi$ (since it is $\fld{T}_\op{A^*}$-invariant), it follows that $\overline{\ssp{V}^\perp\cap\dom{A^*}}=\ssp{V}^\perp$. Also, by invoking Lemma \ref{Lem:SubspaceAdjoint} (item 4)) and the fact that $\overline{\ssp{V}\cap\dom{A}}=\ssp{V}$, we have (note that $\dim(\ima\op{B})<\infty$, and consequently $\ima\op{B} = \overline{\ima\op{B}}$)
	\begin{equation}
	\begin{split}
	\op{A}^*(\ssp{V}^\perp\cap(\ima\op{B})^\perp\cap D(\op{A}^*))\subseteq\ssp{V}^\perp.
	\end{split}
	\end{equation}
	Hence, $\ssp{V}^\perp$ is an ($\op{B}^*$,$\op{A}^*$)-invariant subspace. By invoking Theorem \ref{Thm:CA_T_Inv}, it follows that $\ssp{V}^\perp$ is $\fld{T}$-conditioned invariant with respect to ($\op{B}^*$,$\op{A}^*$), and consequently, by using Lemma \ref{Lem:DualCurtain} it follows that $\ssp{V}$ is $\fld{T}$-controlled invariant.\\
	({\bf Only if part}): Let $\ssp{V}$ be $\fld{T}$-controlled invariant. By invoking Lemma \ref{Lem:DualCurtain}, it follows that  $\ssp{V}^\perp$ is $\fld{T}$-conditioned invariant. Therefore, from Theorem \ref{Thm:CA_T_Inv} it follows that  $\ssp{V}^\perp = \ssp{W}_\psi+\ssp{W}_\mr{f}$, with $ \ssp{W}_\psi$ defined as above and $\dim(\ssp{W}_\mr{f})<\infty$.  Also, since $\dom{A^*}$ is densely defined on $\ssp{V}^\perp$  (from Lemma \ref{Lem:DualCurtain}, we obtain $\ssp{V}^\perp$ is $\fld{T}_\op{A^*}$-invariant, and consequently $\overline{\ssp{V}^\perp\cap\dom{A^*}}=\ssp{V}^\perp$) and $\ssp{W}_\psi$ (since it is $\fld{T}_\op{A^*}$-invariant), one can assume $\ssp{W}_\mr{f}\subset\dom{A^*}$. Hence, $\ssp{V}= \ssp{V}_\phi\cap(\ssp{W}_\mr{f})^\perp$, where $\ssp{V}_\phi = \ssp{W}_\psi^\perp = \overline{\spanset{\ssp{E}_i}_{i\in\fld{I}}}$ and  $\ssp{W}_f\subset\dom{A^*}$. This completes the proof of the theorem.
\end{proof}
\begin{remark}\label{Rem:Lemma12}
	 Below, we emphasize that Theorem \ref{Thm:MyDual} is compatible with the currently available results in the literature. In the literature, the following main results corresponding to $\fld{T}$-controlled invariant subspaces are available.
	\begin{enumerate}
		\item As shown in \cite{Pandolfi_Disturbance_86} (Theorem 3.1) and \cite{SchmidtInv_1980} (Theorem 2.2) the necessary condition for $\fld{T}$-controlled invariant is $\overline{\ssp{V}^\perp\cap\dom{A^*}}=\ssp{V}^\perp$. Since in $\ssp{V}_\mr{f}\subset\dom{A^*}$, this result is compatible with Theorem \ref{Thm:MyDual} (only if part).
		\item  In \cite{Byrnes1998} it is shown that for  single-input single-output (SISO) systems if $c\in\dom{A^*}$ and $<c,b>\neq0$, then the subspace $\ker\op{C}$ is $\fld{T}$-controlled invariant, where $\op{C} = <c,\cdot>$, and the corresponding bounded feedback gain is given by $\op{F} = -\frac{<\op{A}^*c,\cdot>}{<c,b>}$.  Now, we show that this result and Theorem \ref{Thm:MyDual} are consistent. Since $\op{X}=\ker\op{C}+\ima\op{B}$, $\ssp{V}=\ker\op{C}$ is \AB-invariant and consequently feedback \AB-invariant (by invoking Lemma \ref{Lem:AB_Equivalency}). Moreover, $\ssp{V} = \op{X}\cap(\ima\op{C}^*)^\perp$ (note that $\ima\op{C}^*=\mr{span}\{c\}$), and hence from Theorem \ref{Thm:MyDual} (since $\op{X}$ is obtained as sum of all sub-eigenspaces of $\op{A}$ and $c\in\dom{A^*}$, one can set $\ssp{V}_\phi=\op{X}$ and $\ssp{V}_\mr{f}=\mr{span}\{c\}$), $\ssp{V}$ is $\fld{T}$-controlled invariant. In other words, sufficient conditions of Theorem \ref{Thm:MyDual} are also compatible with the result in \cite{Byrnes1998} (for SISO systems).
		\item Note that  $\ssp{V}_\mr{f}\subset\dom{A^*}$ is a crucial condition. Similar to the above analysis, consider a SISO system and the subspace $\ssp{V} = \op{X}\cap(\ima\op{C}^*)^\perp$. Assume that $c\notin\dom{A^*}$, and consequently the feedback introduced in \cite{Byrnes1998} (i.e., $\op{F}_1x = -\frac{<c,\op{A}x>}{<c,b>}$) is not bounded. In fact $\ssp{V}$ is not $\fld{T}$-invariant (since it does not satisfy the necessary condition in \cite{Pandolfi_Disturbance_86} (Theorem 3.1)).
		\begin{enumerate}
			\item It should be pointed out that although one can still construct another bounded feedback $\op{F}_2$ as derived in the proof of Lemma \ref{Lem:AB_Equivalency} so that $\ssp{V}$ is feedback \AB-invariant, however,  even with this bounded feedback, $\ssp{V}$ is not $\fld{T}$-controlled invariant (since $\ssp{V}$ does not satisfy the necessary conditions).
			\item It can be shown that $\op{F}_2$ (as constructed in Lemma \ref{Lem:AB_Equivalency}) is expressed as $\op{F}_2 = \alpha(\op{F}_1^*)^*$, where $\alpha\in\fld{R}$ is determined based on $c$ and $w_1\in\dom{A}$ as used in Lemma \ref{Lem:AB_Equivalency} ($\ssp{W}_f=\spanset{w_1}$). Therefore, even if $\op{F}_1$ is unbounded,  $\op{F}_1^*:\fld{R}\rightarrow\op{X}$ is bounded (since it is defined on the Fin-D vector space $\fld{R}$), and consequently $(\op{F}_1^*)^*$ is bounded.
		\end{enumerate}
		
	\end{enumerate}
	
\end{remark}
	
\section{Fault Detection and Isolation (FDI) Problem}\label{Sec:FDI}
In this section,  we first formulate the FDI problem for the RS system \eqref{Eq:GeneralID} and then the methodology that was developed in the previous section is utilized to derive and provide  necessary and sufficient conditions for solvability (formally defined in Remark \ref{Rem:FDIName1}) of the FDI problem.
\subsection{The FDI Problem Statement}\label{Sec:FDIProblem}
Consider the following regular RS system
\begin{equation}\label{Eq:FaultySys}
\begin{split}
\dot{x}(t)&=\op{A}x(t)+\op{B}u(t)+\sum_{i=1}^p\op{L}_if_i(t),\\
y(t) &= \op{C}x(t),
\end{split}
\end{equation}
where $\op{L}_i$'s and $f_i$'s ($f_i(t)\in\fld{R}$) denote the fault signatures and signals, respectively. The other variables and operators are defined as in the model \eqref{Eq:GeneralID}.
The FDI problem is specified in terms of generating a set of residual signals, denoted by $r_i(t)$ , $i=1,\cdots,p$ such that each residual signal $r_i(t)$ is decoupled from the external input and all the faults, except one fault $f_i(t)$. In other words, the residual signal $r_i(t)$ satisfies the following conditions for all $u(t)$ and $f_j$ ($j\neq i$)
\begin{subequations}\label{Eq:FDICond}
	\begin{align}
	\mathrm{if}&\; f_i= 0 \; \Rightarrow r_i \rightarrow 0\;(\text{stability\; and\; decoupling\; condition}), \label{Eq:FDICondDecouple}\\
	\mathrm{if}&\; f_i\neq 0 \;\Rightarrow r_i\neq 0 .\label{Eq:FDICondCouple}
	\end{align}
\end{subequations}

The residual signal $r_i(t)$ is to be generated from the following dynamical detection filter
\begin{equation}\label{Eq:Filter}
\begin{split}
\dot{\omega}_i(t) &=\op{A}_o\omega_i(t)+ \op{B}_o u(t)+ \op{E}_iy(t),\\
r_i(t) &= H_iy(t) - \op{M}_i\omega_i(t),
\end{split}
\end{equation}
where $\omega_i\in\op{X}_o^i$, $\op{X}_o^i$ is a separable Hilbert space (Fin-D or Inf-D), and  $\op{A}_o$  is a regular RS operator.  The operators  $\op{B}_o$, $\op{E}_i$, $\op{M}_i$ and $H_i$ are closed operators with appropriate domains and codomains (for example, $\op{A}_o:\op{X}_o^i\rightarrow\op{X}_o^i$ and $\op{E}_i:\fld{R}^q\rightarrow\op{X}_o^i$). In this work, we investigate, develop, and derive conditions for constructing the detection filter \eqref{Eq:Filter} by utilizing invariant subspaces such that the condition \eqref{Eq:FDICond} is satisfied.

\begin{remark}\label{Rem:FDIName1}
Design of the detection filter \eqref{Eq:Filter} involves satisfying two main requirements:
	\begin{enumerate}
		\item The residual signal $r_i(t)$ should be decoupled from all faults except $f_i(t)$.
		\item The corresponding filter error dynamics (where error is defined as  the difference between the detection filter state and the corresponding RS system state) should be stable.
	\end{enumerate}
	If the first requirement is satisfied, we say that the fault $f_i$ is \underline{detectable and isolable}. However, the FDI problem is said to be \underline{solvable} if \underline{both} requirements are simultaneously satisfied.
\end{remark}


In the next subsection, we  derive necessary and sufficient conditions  for solvability of the FDI problem for the RS system \eqref{Eq:FaultySys}.
\subsection{Necessary and Sufficient Conditions}
As stated above, the FDI problem can be cast as that of designing dynamical detection filters having the structure \eqref{Eq:Filter} such that each detection filter output is  decoupled from  all faults but one. By augmenting the RS system \eqref{Eq:FaultySys} and the detection filter \eqref{Eq:Filter}, one can obtain the representation
\begin{equation}\label{Eq:AugSys}
\begin{split}
\dot{x}^e(t) &= \op{A}^ex^e(t) + \op{B}^eu(t) +\sum_{i=1}^p \op{L}_i^ef_i(t),\\
r_i(t) &= \op{C}^e x^e(t),
\end{split}
\end{equation}
where $x^e(t) = \bbm x\\ \omega_i\ebm\in\op{X}^e=\op{X}\oplus\op{X}_o^i$, $\op{C}^e = \bbm H_i\op{C} &-\op{M}_i\ebm$ and
\begin{equation}\label{Eq:AugSysPar}
\begin{split}
\op{A}^e &= \bbm \op{A} &0\\ \op{E}_i\op{C} &\op{A}_o\ebm\;,\; \op{B}^e = \bbm \op{B}\\ \op{B}_o\ebm,\;\;
\op{L}_i^e = \bbm \op{L}_i\\0\ebm.
\end{split}
\end{equation}

First, let us present the following important lemma.
\begin{lemma}\label{Lem:AugSemigroup}
	Assume that the operators $\op{A}_{11}:\op{X}_1\rightarrow\op{X}_1$ and $\op{A}_{22}:\op{X}_2\rightarrow\op{X}_2$ are  infinitesimal generators of two $C_0$ semigroups $\fld{T}_{\op{A}_{11}}$ and $\fld{T}_{\op{A}_{22}}$, respectively. Let the operator $\op{A}_{21}:\op{X}_1\rightarrow\op{X}_2$ be bounded. Then
	\begin{enumerate}[label=(\alph*)]
		\item   $\op{A}_{e}= \bbm \op{A}_{11} &0\\ \op{A}_{21} &\op{A}_{22}\ebm$ is an infinitesimal generator of the following $C_0$ semigroup in $\op{X}_e=\op{X}_1\oplus\op{X}_2$
		\begin{align*}
		\fld{T}_{\op{A}} &= \bbm \fld{T}_{\op{A}_{11}} &0\\ \fld{T}_{\op{A}_{21}} &\fld{T}_{\op{A}_{22}}\ebm,\;
		\fld{T}_{\op{A}_{21}}(t)x = \int_0^t \fld{T}_{\op{A}_{22}}(t-s)\op{A}_{21}\fld{T}_{\op{A}_{11}}x ds.
		\end{align*}
		\item Moreover, if $\op{A}_{11}$ and $\op{A}_{22}$ are regular RS operators with only finitely many common eigenvalues,  then $\op{A}_{e}$ is also a regular RS operator.
	\end{enumerate}
\end{lemma}
\begin{proof}
	(a) This follows from the Proposition 4.7 in \cite{DynFeedbackInf_Book}.\\
	(b) We first show that the operator $\op{A}_d = \bbm\op{A}_{11} &0\\ 0 &\op{A}_{22}\ebm$ is a regular RS with a finitely many multiple (repeated) eigenvalues. It can be shown that $\lambda$ is an eigenvalue of $\op{A}_d$ if and only if $\lambda$ is an eigenvalue of $\op{A}_{11}$ or $\op{A}_{22}$. Hence, $\op{A}_d$ is an operator with finitely many multiple (repeated) eigenvalues. Moreover, each generalized eigenvector of $\op{A}_d$ can be expressed as $\bbm \phi_1\\ 0\ebm$ or $\bbm 0\\ \phi_2\ebm$, where $\phi_1$ and $\phi_2$ denote the generalized eigenvectors of $\op{A}_{11}$ and $\op{A}_{22}$, respectively. It follows that $\ssp{P}_i^d = \op{Q}\ssp{P}_i$ (where $\ssp{P}_i$ is an eigenspace of the operator $\op{A}_{11}$ and $\op{Q}$ is an embedding operator such that $\op{Q}:\op{X}_1\rightarrow\op{X}_e$ and $\op{Q}x= \bbm x\\0\ebm$) is an eigenspace of $\op{A}^e$. Furthermore, the same result holds for the eigenspaces of $\op{A}_{22}$. Hence, it can be shown that the condition (3) in Definition \ref{Def:RegularRS} is satisfied.  Finally, we show that the inequality that is defined in Remark \ref{Rem:CondOnDefRegRSOp} holds. If $\lambda_i\in\sigma(\op{A}_{11})\cap\sigma(\op{A}_{22})$, we select $d_i = \min(\inf_{\lambda\in\sigma(\op{A}_{11})-\lambda_i}|\lambda-\lambda_i|, \inf_{\lambda\in\sigma(\op{A}_{22})}|\lambda-\lambda_i|)$. Since the number of common eigenvalues of $\op{A}_{11}$ and $\op{A}_{22}$ is finite, it follows that the inequality in Remark \ref{Rem:CondOnDefRegRSOp} is satisfied, and consequently $\op{A}_d$ is a regular RS with a finitely many multiple (repeated) eigenvalues. Given that the operator $\bbm 0 &0\\\op{A}_{21} &0\ebm$ is bounded (with a bound equal to the bound of $\op{A}_{21}$), and  by invoking Remark \ref{Rem:CondOnDefRegRSOp}, it follows that the operator $\op{A}_{e}$ is a regular RS operator. This completes the proof of the lemma.
\end{proof}

Note that $\op{A}_o$ in  \eqref{Eq:Filter} is assumed to be a regular RS operator and the operator $\op{E}$ (and consequently $\op{EC}$) is a bounded operator. If $\op{A}_o$ and $\op{A}$ have only finitely many common eigenvalues,  by invoking Lemma \ref{Lem:AugSemigroup} it follows that  $\op{A}^e$, as per equation  \eqref{Eq:AugSysPar}, is an infinitesimal generator of a $C_0$ semigroup, and also a regular RS operator.
Next,  we need to establish an important relationship between the unobservable subspace of the system \eqref{Eq:AugSys} and the unobservability subspace of the system \eqref{Eq:FaultySys} as shown in the following lemma.

\begin{lemma}\label{Lem:AugUnobserSpace2Small}
	Consider the augmented system \eqref{Eq:AugSys} and let $\ssp{N}^e=<\ker \op{C}^e|\fld{T}_{\op{A}^e}>$. Then, $\op{Q}^{-1}\ssp{N}^e$ is an unobservability subspace of the system \eqref{Eq:FaultySys}, where $\op{Q}$ is the embedding operator.
\end{lemma}
\begin{proof}
	Let $\op{S} =\op{Q}^{-1}\ssp{N}^e$, where $\op{Q}$ is the embedding operator as defined above. We first show that $\op{S}$ is an \CA-invariant subspace of the system \eqref{Eq:FaultySys} (that is, $\op{A}(\op{S}\cap\ker\op{C}\cap\dom{A})\subseteq\op{S}$). Let us show that $\overline{\op{S}\cap\dom{A}}=\op{S}$. Since $\ssp{N}^e$ is $\fld{T}_{\op{A}^e}$-invariant, we have $\overline{\ssp{N}^e\cap D(\op{A}^e)}=\ssp{N}^e$. Assume that $\overline{\op{S}\cap\dom{A}}\neq\op{S}$, and consequently there exits $x\in\op{S}$ and a neighborhood $B\ni x$ such that $B\cap\dom{A}=0$. It follows that $\op{Q}B\cap D(\op{A}^e)=0$ (note that $\op{Q}x = \bbm x\\0\ebm$) which is in contradiction with the fact that $\overline{\ssp{N}^e\cap D(\op{A}^e)}=\ssp{N}^e$.  Hence, $\overline{\op{S}\cap\dom{A}}=\op{S}$. Now, let $x\in(\op{S}\cap\ker\op{C}\cap\dom{A})$. Since $\ssp{N}$ is $\op{A}^e$-invariant, one can write $\op{A}^e\bbm x\\0\ebm = \bbm \op{A}x\\0\ebm\in\ssp{N}^e$. Therefore, $\op{A}x\in\op{S}$ (i.e., $\op{S}$ is \CA-invariant), and consequently $\op{S}$ is a feedback \CA-invariant subspace (according to Lemma \ref{Lem:CA_equalency}).
	
	We now show that $\op{S}$ satisfies the conditions in Theorem \ref{Thm:CA_T_Inv}.  Since $\ssp{N}^e$ is $\fld{T}_{\op{A}^e}$-invariant and $\op{A}^e$ is a regular RS operator, following the Corollary \ref{Col:UnObs_Span}  we have $\ssp{N}^e = \overline{\spanset{\ssp{E}^e_i}_{i\in\fld{I}}}$, where $\ssp{E}^e_i$'s denote the sub-eigenspaces of $\op{A}^e$. There are three possibilities for a sub-eigenspace of $\op{A}^e$ as follows:
	\begin{enumerate}
		\item $\ssp{E}_i^e=\bbm \ssp{E}_i\\0\ebm$, where  $\ssp{E}_i$ is a sub-eigenspace of $\op{A}$.
		\item $\ssp{E}_i^e=\bbm 0\\\ssp{E}_i^o\ebm$, where $\ssp{E}_i^o$ is a sub-eigenspace of $\op{A}_o$.
		\item $\ssp{E}_i^e = \bbm \ssp{E}_i\\ \ssp{E}_o\ebm$, such that $\ssp{E}_i$ and $\ssp{E}_o$ are \underline{not} sub-eigenspaces of $\op{A}$ and $\op{A}_o$ (this sub-eigenspace corresponds  to the common eigenvalues of $\op{A}$ and $\op{A}_o$).
	\end{enumerate}	
	Let $\op{S}_\phi$ denote the largest subspace in the form $\op{S}_\phi =\overline{\spanset{\ssp{E}_i}_{i\in\fld{I}}}$ such that $\ssp{E}_i$ is a sub-eigenspace of $\op{A}$ that is contained in $\ker H\op{C}$. It follows that $\op{S}_\phi\subseteq\op{S}$ and $\op{S} = \overline{\op{S}_\phi+\op{S}_\mr{f}}$, where $\op{S}_\mr{f}$ is a sum of the sub-eigenspaces in the form of item 3). Since there are only finitely many common eigenvalues of $\op{A}$ and $\op{A}_o$, it follows that $\op{S}_\mr{f}$ is Fin-D. Therefore,  $\op{S}$ satisfies the condition of Theorem \ref{Thm:CA_T_Inv}, and consequently $\op{S}$ is $\fld{T}$-conditioned invariant.
		
	Finally, given that $\op{S}\subseteq\ker H\op{C}$ and $\ssp{N}^e$ is the largest $\fld{T}_{\op{A}^e}$-invariant subspace in $\ker\op{C}$, it follows that $\op{S}$ is the largest $\fld{T}$-conditioned invariant subspace contained in $\ker H\op{C}$ (i.e., $\op{S}$ is an unobservability subspace of the RS system \eqref{Eq:FaultySys}).  This completes the proof of the lemma.
\end{proof}
To clarify of existence of the subspace $\op{S}_\mr{f}$ in the above proof,  consider the following Fin-D example.

\paragraph{Example 1}\ Let us assume that $\op{A}^e$ is given by
\begin{equation}
A^e = \bbm 1 &1 &0 &0\\ 0 &1 &0 &0\\ 0 &1 &1 &0\\ 0 &0 &0 &3\ebm
\end{equation}
Also, let $L_1 =[0,1,0,0]^\tran$ and $L_2 = [1,1,1,0]^\tran$. It follows that $L_2 = A^eL_1$ and $A^eL_2 = 2L_2-L_1$. Therefore, $\ssp{E} = \spanset{L_1,L_2}$ is a sub-eigenspace of $A^e$ (corresponding to $\lambda=1$). However, $Q^{-1}\ssp{E} = \ssp{L}_1$ is not a sub-eigenspace of $A = \bbm 1 &1\\ 0&1\ebm$. This example highlights the reason why we consider $\op{S}_\mr{f}$ in the proof of the above Lemma.

In order to provide  sufficient conditions for solvability of the FDI problem, one also needs to show that the error dynamics corresponding to the designed fault detection observer is stable. The following theorem  provides  necessary and sufficient conditions for  stability of a general Inf-D system.
\begin{lemma}\label{Lem:Stability}(\cite{Curtain_Book} - Theorem 5.1.3)
	Consider the Inf-D system $\dot{e}(t) = \op{A}_e e(t)$, such that $\op{A}_e$ is an infinitesimal generator of a $C_0$ semigroup. The system is exponentially stable if and only if there exists a positive definite and bounded operator $\op{P}_e:\op{X}\rightarrow\op{X}$ such that
	\begin{equation}\label{Eq:ErrorStabilityCondition}
	<\op{A}_ez,\op{P}_ez> + <\op{P}_ez,\op{A}_ez> = -<z,z>,\; \forall z\in D(\op{A}_e)
	\end{equation}
	
\end{lemma}

We are now in the position to derive the solvability necessary and sufficient conditions for the FDI problem corresponding to the RS system \eqref{Eq:FaultySys}.
\begin{theorem}\label{Thm:General_Nes_Suff_Cond}
	Consider the regular RS system \eqref{Eq:FaultySys}. The FDI problem has a solution only if
	\begin{equation}\label{Eq:NecessaryCondition}
	\op{S}_i^*\cap\ssp{L}_i = 0,
	\end{equation}
	where $\op{S}_i^* = <\ker H_i\op{C}|\op{A+D}_i\op{C}>$ is the smallest unobservability subspace containing $\ssp{L}_j$, where $j=1,\cdots,p$ and $j\neq i$, and $\ssp{L}_i=\spanset{\op{L}_i}$. On the other hand, if the above condition is satisfied  and there exist two maps $\op{D}_o$ and $\op{P}_e$ such that $(\op{A}_{p}+\op{D}_o\op{M}_i)$ and $\op{P}_e$ satisfy the condition \eqref{Eq:ErrorStabilityCondition}, then the FDI problem is solvable where $\op{A}_{p} = (\op{A+D}_i\op{C})|_{\op{X}/\op{S}_i^*}$ (i.e., $\op{A}_{p}$ is the operator induced by $\op{A+D}_i\op{C}$ on the factor space $\op{X}/\op{S}_i^*$), and $\op{M}_i$ is the solution to $\op{M}_i\op{P}_i=H_i\op{C}$, where $\op{P}_i$ is the canonical projection from $\op{X}$ onto $\op{X}/\op{S}_i^*$.
\end{theorem}
\begin{proof}
	({\bf Only if part}): We consider, without loss of generality, that  the system \eqref{Eq:FaultySys} is subject to two faults $f_1$ and $f_2$. Assume that the detection filter \eqref{Eq:Filter} is designed such that the residual (that is, the output of the filter) is decoupled from the fault $f_2$ but requires to be sensitive to the fault $f_1$. By considering the augmented system \eqref{Eq:AugSys}, it is necessary that $\ssp{L}_{2}^e=\spanset{\op{L}_2^e}\subseteq\ssp{N}^e$, ($\op{L}_2^e$ is defined in \eqref{Eq:AugSysPar}) where $\ssp{N}^e$ is the unobservable subspace of \eqref{Eq:AugSys}. By invoking Lemma \ref{Lem:AugUnobserSpace2Small}, the subspace $\op{S}=\op{Q}^{-1}\ssp{N}^e$ is an unobservability subspace of the pair ($\op{C}$,$\op{A}$) containing $Q^{-1}\ssp{L}_2^e=\ssp{L}_2$. Moreover, in order to detect the fault $f_1$ (which can be an arbitrary function of time), it is necessary that $\ssp{N}^e\cap\ssp{L}_1^e=0$ . Hence, $\op{S}\cap\ssp{L}_1 = 0$. Since $\op{S}^*_1$ is the minimal unobservability subspace containing $\ssp{L}_2$ (i.e., $\op{S}^*_1\subseteq\op{S}$), the necessary condition for satisfying the above condition is $\op{S}^*_1\cap\ssp{L}_1 = 0$. \\
	({\bf If part}): Assume that $\op{S}_1^*\cap\ssp{L}_1 = 0$, and let $\op{D}_1$ and $H_1$ be defined according to  $\op{S}_1^*$ (refer to the Definition \ref{Def:UnobservabilitySpace}). By definition, $\op{L}_2\subseteq\op{S}_1^*$ where $\op{S}_1^*$ is the unobservable subspace of the system ($H_1\op{C}$,$\op{A+D}_1\op{C}$). In other words, $\op{S}^*_1 = <\ker H_1\op{C}|\op{A+D}_1\op{C}>$.
	
	Now consider the canonical projection $\op{P}_1:\op{X}\rightarrow\op{X}/\op{S}_1^*$ and the following detection filter
	\begin{equation}\label{Eq:DetectionFilter}
	\begin{split}
	\dot{\omega}_1(t) = &\op{F}_1\omega_1(t) + \op{G}_1u(t) - \op{E}_1y(t)\\
	r_1(t) = &\op{M}_1\omega_1(t) - H_1y(t)
	\end{split}
	\end{equation}
	where $\op{F}_1 = \op{A}_p+\op{D}_o\op{M}_1$, $\op{G} = \op{P}_1\op{B}$ and $\op{E}_1 = \op{D}_1+\op{P}_1^{-r}\op{D}_oH_1$.  By defining the error signal as $e(t) = \op{P}_1x(t)-\omega_1(t)$, one can obtain
	\begin{equation}\label{Eq:ErrorDyn}
	\begin{split}
	\dot{e}(t) &= \op{F}_1e(t) + \op{P}_1\op{L}_{1}f_1(t),\\
	r_1(t) &=\op{M}_1e(t).
	\end{split}
	\end{equation}
	By invoking  Lemma \ref{Lem:Stability}, it follows that the error dynamics \eqref{Eq:ErrorDyn} is exponentially stable. Therefore, if $f_1\equiv0$ (for any value of $f_2$) then  $r_{1}(t)\rightarrow 0$. Otherwise, $||r_{1}(t)||\neq 0$ (which can be used for declaring the detection of the fault $f_1$). This completes the proof of the theorem.
\end{proof}

\begin{remark}\label{Rem:FDIName}
	Note that the FDI problem was solved by designing a fault detection filter to estimate $x_1$. However, unlike the Fin-D case, the condition $\ssp{N}=0$ (the unobservable subspace) is not sufficient for the existence of an observer for a general Inf-D system \cite{DynFeedbackInf_Book}. Therefore, the condition \eqref{Eq:NecessaryCondition} is not sufficient for solvability of the FDI problem, and therefore one needs the extra condition that is stated in Theorem \ref{Thm:General_Nes_Suff_Cond}.
\end{remark}

\subsection{Solvability of the FDI Problem Under Two Special Cases}\label{Sec:FDISpecialCase}
In this subsection, we investigate two special cases, where the condition \eqref{Eq:NecessaryCondition} provides a \emph{single} necessary and sufficient condition for solvability of the FDI problem.
\subsubsection{\underline{{\bf Case 1}}}\ \\
The following theorem provides a necessary and sufficient condition for  solvability of the FDI problem when the number of positive eigenvalues of the quotient subsystem is finite.
\begin{theorem}\label{Thm:SpecialCase1FDICond}
	Consider the faulty RS system \eqref{Eq:FaultySys} with $\op{C}$ specified as in equation \eqref{Eq:OuputOperator}, and let the operator $(\op{A+DC})$ have only finite number of positive eigenvalues and the operator $(\op{A+DC})|_{\op{X}/\op{E}^+}$ be asymptotically stable, where $\op{E}^+$ is the sum of eigenspaces corresponding to the positive eigenvalues. The FDI problem is solvable if and only if the condition \eqref{Eq:NecessaryCondition} holds.
\end{theorem}
\begin{proof}
	({\bf if part}): Consider the detection filter \eqref{Eq:DetectionFilter}. As stated above, the observer gain $\op{D}_o$ is designed such that the operator $\op{A}_p+\op{D}_o\op{M}_1$ is asymptotically stable.
	Given that the unobservable subspace of the system ($\op{M}_1$, $\op{A}_p$) is zero (since it is obtained by factoring out $\op{S}_1^*$), the  Fin-D pair ($M_1^+$, $A_{p}^+$) (that are induced from $\op{M}_1$ and $\op{A}_p$ on $\op{X}_1^+$) is observable. Therefore, there exists an operator ${D}_o:\fld{R}^{q_h}\rightarrow\op{X}_1^+$ ($q_h = \rank(H_1\op{C})$) such that all the eigenvalues of $A_{p}^++D_oM_1^+$ are negative. By invoking  the asymptotic stability of $\op{A}_{p}^-$, and considering $\op{D}_o$ as the extension of $D_o$, one can show that the error dynamics \eqref{Eq:ErrorDyn} is asymptotically stable. By following along the same lines as in the proof of Theorem \ref{Thm:General_Nes_Suff_Cond}, it follows that the FDI problem is solvable.\\
	({\bf only if part}): This follows from the results that are stated in Theorem \ref{Thm:General_Nes_Suff_Cond}.\\
	This completes the proof of the theorem.
\end{proof}
\subsubsection{\underline{{\bf Case 2}}}\ \\
In this case, the faulty RS system \eqref{Eq:FaultySys} is specified according to the operator given by equation \eqref{Eq:OuputOperator}, however $c_i$'s are governed and restricted to
\begin{equation}\label{Eq:NewOuputOperator}
\begin{split}
c_i = \sum_{i=1}^{n_c}\zeta_{i,j}\psi_j.
\end{split}
\end{equation}
In other words, the $c_i$ vectors lie on a finite dimensional subspace of $\op{X}$. Since $<\phi_i,\psi_j> = \delta_{ij}$, it follows that $\op{C}\phi_i=0$ for all $i>n_c$. Therefore, $\mathrm{span}\SRB{\phi}{n_c+1}\subseteq\ker\op{C}$, and consequently, $\ker \op{C} = \overline{\ssp{C}^0_f\oplus\mathrm{span}\SRB{\phi}{n_c+1}}$, where $\ssp{C}^0_f\subseteq\spanset{\{\phi_j\}_{j=1}^{n_c}}$. By invoking Lemma \ref{Lem:SumClosedSpace} and the fact that  $\dim(\ssp{C}^0_f)<\infty$, we have   $\ker\op{C}=\ssp{C}^0_f\oplus\overline{\mathrm{span}\SRB{\phi}{n_c+1}}$.
Since every $\overline{\SRB{\phi}{n_c+1}}\subseteq\ker\op{C}$ is also $\SG{T}{A+DC}$-invariant and contained in $\ker H\op{C}$, it follows that the unobservability subspace $\op{S}$ containing a given subspace $\ssp{L}$ necessarily contains the Inf-D subspace $\overline{\SRB{\phi}{n_c+1}}$. Therefore, the factored out quotient subsystem ($\op{M}_1$, $\op{A}_p$) is Fin-D and one can provide necessary and sufficient conditions for solvability of the FDI problem. The following theorem summarizes this result.
\begin{theorem}\label{Thm:SpecialCase2FDICond}
	Consider the faulty system \eqref{Eq:FaultySys} that is assumed to be an RS system and specified according to the output operator \eqref{Eq:NewOuputOperator}. The FDI problem is solvable if and only if $\op{S}^*_i\cap\ssp{L}_i = 0$, where $\op{S}^*_i$ is the smallest unobservability subspace containing $\ssp{L}_j$, $j=1,\cdots,p$ and $j\neq i$.
\end{theorem}
\begin{proof}
		({\bf if part}): Note that $\op{X}/\op{S}_1^*$ is a Fin-D vector space and the system ($M_1$, $A_{p}$) (where $A_{p} = (\op{A+D}_1\op{C})|_{\op{S}_1^*}$ and $M_1\op{P}_1 = H_1\op{C}$) is observable and Fin-D. Therefore, there always exists the operator $\op{D}_o$ such that the observer \eqref{Eq:Filter} can both detect and isolate the fault $f_i$. Given that the detection filter is Fin-D, the stability of the error dynamics is guaranteed by the observability of the system ($M_1$, $A_{p}$). \\
		({\bf only if part}): This follows from the results that are stated in Theorem \ref{Thm:General_Nes_Suff_Cond}.\\
		This completes the proof of the theorem.
\end{proof}

\subsection{Summary of Results}
In this section, the FDI problem was formulated by invoking invariant subspaces that were introduced and developed in Section \ref{Sec:InvSpace}. We first derived in Theorem \ref{Thm:General_Nes_Suff_Cond}  necessary  and sufficient conditions  for solvability of the FDI problem. Moreover, it was shown that for two special classes of  regular RS systems there exists a \emph{single} necessary and sufficient condition (that is, the condition \eqref{Eq:NecessaryCondition}) for solvability of the FDI problem.
Table \ref{Tab:Procedure_RS} summarizes and provides a pseudo-code and procedure for detecting and isolating faults in the RS system \eqref{Eq:FaultySys}.
\begin{remark}
	As illustrated above, the main difficulty in deriving a single necessary and sufficient condition for solvability of the FDI problem for a regular RS system has its roots in the relationship between the condition $\ssp{N}=0$ and the existence of a bounded observer gain $\op{D}$ such that the corresponding error dynamics is exponentially stable. Another possible approach that one can investigate and pursue is through a frequency-based approach that was originally developed in \cite{Zwart_Book} to investigate the disturbance decoupling problem. This approach  deals with the Hautus test, and as shown in \cite{JacobZwart2004} the Hautus test does also involve certain difficulties for Inf-D systems. Specifically,  there exist certain Inf-D systems that pass the Hautus test, however they are not observable. Notwithstanding the above, the investigation of  utilizing a frequency-based approach for tackling the FDI problem and its relationship with invariant subspaces that are introduced in our work is beyond the scope of this paper, and therefore we suggest this line of research as part of our future work.
\end{remark}
\begin{table}[h]
	\caption{Pseudo-algorithm for detecting and isolating the fault $f_i$ in the regular RS system \eqref{Eq:FaultySys}.}
	\centering
	\begin{tabularx}{1\columnwidth}{|X|}
		\hline
		\hline
		\begin{enumerate}
			\item {\tt Compute the minimal conditioned invariant subspace $\ssp{W}^*$ containing all $\op{L}_j$ subspaces such that $j\neq i$ (by using the algorithm \eqref{Eq:T-CondSpaceAlg} where $\ssp{L}=\sum_{j\neq i} \op{L}_j$).}
			\item {\tt Compute the unobservability subspace $\op{S}_i^*$ containing $\sum_{j\neq i} \op{L}_j^1$ (by using the algorithm \eqref{Eq:T-UnobSpaceAlg}).}
			\item {\tt Compute the operator $\op{D}_i$ such that $\op{D}_i\in\underline{\op{D}}(\ssp{W}^*)$.}
			\item {\tt Find the operator $H_i$ such that $\ker H_i\op{C} = \overline{\ssp{W}^*+\ker\op{C}} = \overline{\op{S}_i^*+\ker C}$.}
			\item {\tt If $\op{S}_i^*\cap\op{L}_i = 0$,  then the necessary condition for solvability of the FDI problem is satisfied. Moreover, if one of the following conditions are satisfied, the FDI problem is solvable. In other words, one can design a detection filter according to the structure provided in \eqref{Eq:Filter} to detect and isolate $f_i$,}
			\begin{itemize}
				\item {\tt If there exists a bounded operator $\op{D}_o$ such that the conditions of Theorem \ref{Thm:General_Nes_Suff_Cond} are satisfied, or}
				\item {\tt The operator $\op{A}_{p} = (\op{A+D}_i\op{C})|_{\op{X}/\op{S}_i^*}$  has finite number of positive eigenvalues, or}
				\item {\tt If $\dim(\op{X}/\op{S}^*_i)<\infty$.}
			\end{itemize}
			{\tt The operators in the detection filter \eqref{Eq:Filter} are defined as follows. Let $\op{P}_i$ denote the canonical projection of $\op{S}_i^*$, then $\op{A}_o = (\op{A+D}_i\op{C})|_{\op{X}/\op{S}_i^*}+\op{D}_o\op{M}_i$, $\op{B}_o=\op{P}_i\op{B}$,  $\op{M}_i\op{P}_i=H_i\op{C}$, $\op{E}_i=\op{D}_oH_i$ and $\op{D}_o$ is selected such that $\op{A}_o$ satisfies the condition of Lemma \ref{Lem:Stability}. Moreover, the output of the detection filter (i.e., $r_i(t)$) is the residual that satisfies the condition \eqref{Eq:FDICond}.}
		\end{enumerate}\\
		\hline
	\end{tabularx}
	\label{Tab:Procedure_RS}
\end{table}

Finally, to add further clarification and information we have provided in Figure \ref{Fig:Struct} a schematic summarizing and depicting the relationships among the various lemmas, theorems and corollaries that are presented and developed in this paper.
\begin{figure*}
	\centering
	\includegraphics[scale=0.75]{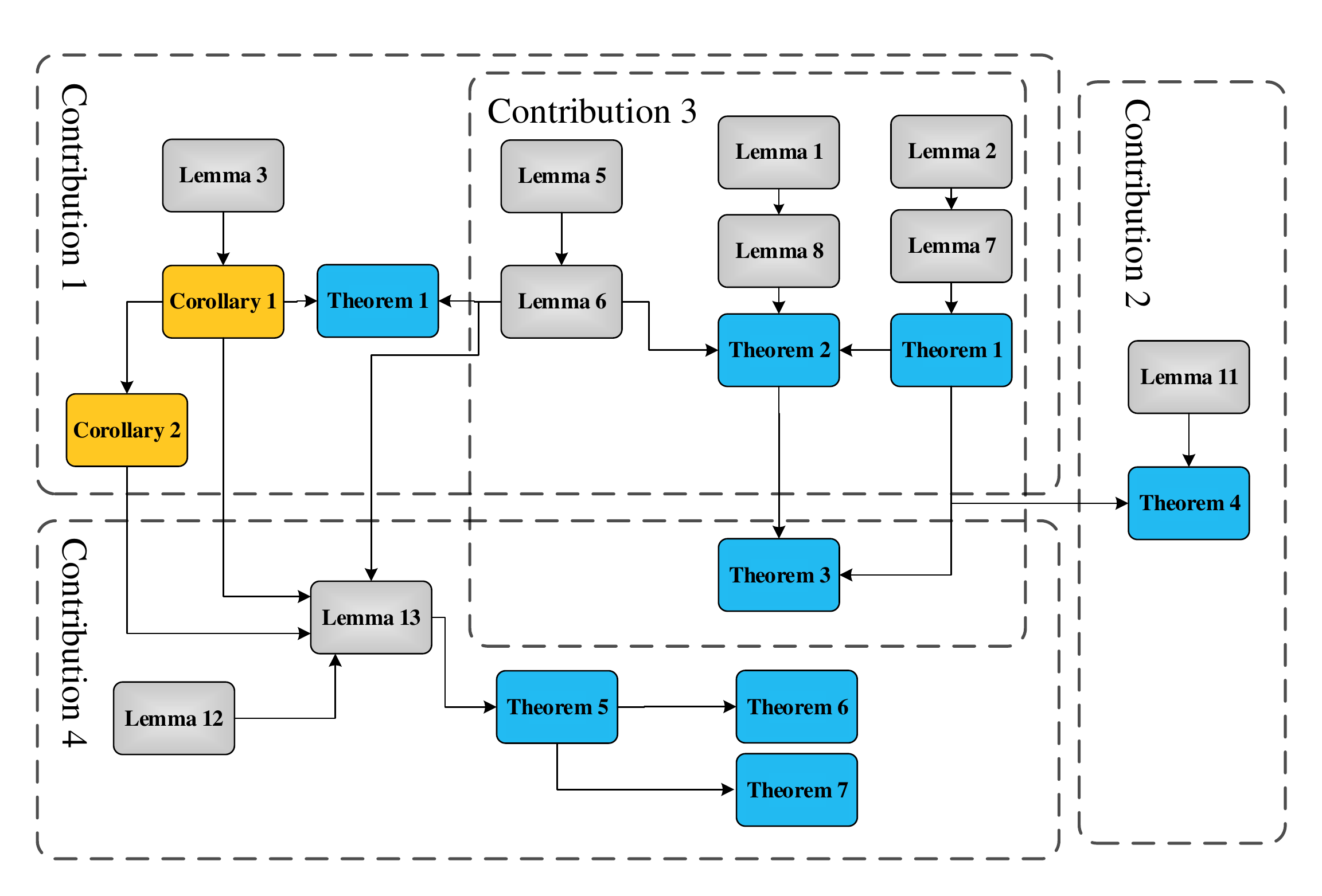}
	\caption{The flowchart  depicting the relationships among lemmas, theorems and corollaries that are developed and presented in this paper.}
	\label{Fig:Struct}
\end{figure*}

\section{Numerical Example}\label{Sec:SimResult}
In this section, we provide a numerical example to demonstrate the applicability of our proposed approach. Consider the following parabolic PDE system
\begin{align}\label{Eq:PDE}
\bbm \pardiff{\tilde{x}_1(t,z)}{t}\\  \pardiff{\tilde{x}_2(t,z)}{t}\ebm = &\bbm \frac{\partial^2 }{\partial z^2} &0.1\\ 0.1 &\frac{\partial^2 }{\partial z^2}\ebm\bbm \tilde{x}_1(t,z)\\ \tilde{x}_2(t,z)\ebm + b_1(z)\tilde{u}_1(t,z)+b_2(z) \tilde{u}_2(t,z)+L_1(z)\tilde{f}_1(t,z)\notag\\&+L_2(z) \tilde{f}_2(t,z)+\bbm \nu_1(t,z)\\ \nu_2(t,z)\ebm,\notag\\
y_1(t) =& \int_{0}^{\pi}c_1(z)\tilde{x}(t,z)\text{d}z+ w_1(t,z), \tilde{x}_i(t,0) =0,\; i=1,2,\\ y_2(t) =& \int_{0}^{\pi}c_2(z)\tilde{x}(t,z)\text{d}z+w_2(t,z), \frac{\partial \tilde{x}_i(t,0)}{\partial z}=0,\; i=1,2,\notag
\end{align}
where $\tilde{x}(t,z) = [\tilde{x}_1(t,z),\tilde{x}_2(t,z)]^\tran\in\fld{R}^2$ and $\tilde{u}_i(t,z)\in\fld{R}$ denote the state and input, respectively. Also, $z\in[0,\pi]$ denotes the spatial coordinate, and $c_i\in \fld{L}_2([0,\pi])^2$, where $\fld{L}_2([0,\pi])$ denotes the space of all square integrable functions over $[0,\pi]$. Also $\nu_i$'s and $w_i$'s ($i=1,2$) denote the process and measurement noise that are assumed to be normal distributions with 0.5 and 0.2 variances, respectively.

It should be pointed out that  the PDE system \eqref{Eq:PDE} represents a linearized approximation to  the model that corresponds to a large class of chemical processes, such as the two-component reaction-diffusion process (for more detail refer to \cite{ChrisBook}). Moreover, the faults $f_1$ and $f_2$ represent malfunctions in the heat jackets (these jackets are modeled by invoking the input vectors $b_1$ and $b_2$).

The system \eqref{Eq:PDE} can be expressed in the representation of \eqref{Eq:FaultySys}  by utilizing the spectral operator $\op{A} = \bbm \frac{\partial^2 }{\partial z^2} &0.1\\ 0.1 &\frac{\partial^2 }{\partial z^2}\ebm$ (and neglecting the disturbances and noise signals $\nu_i$ and $w_i$), where the domain of $\op{A}$ is defined by \cite{Curtain_Book} (Chapter 1):
$$D(\op{A})=\{x\in \bsm\fld{L}_2([0,\pi])^2\esm|\; x,\frac{\mr{d}x}{\mr{d}z}\;\mathrm{are}\;\mathrm{absolutely}\; \mathrm{continuous}\}.$$
By solving the corresponding Sturm-Liouville problem \cite{KreysziG_Book}, the eigenvalues of $\op{A}$ are obtained as $\lambda_k^1 = 0.1-k^2, \lambda_k^2 = -0.1-k^2, k\in\fld{N}$, and the corresponding eigenfunctions are given by $\phi_k^1 = \sqrt{\frac{2}{\pi}}[\sin(kz), \sin(kz)]^\tran$ and $\phi_k^2 = \sqrt{\frac{2}{\pi}}[\sin(kz), -\sin(kz)]^\tran$. Moreover, $\psi_k = \sqrt{\frac{2}{\pi}}\cos(kz)$'s are bi-orthogonal functions.
Consider the system \eqref{Eq:PDE}, where $c_1(z) = \begin{cases}[1,1]^\tran &;\;0\leq z\leq \pi/4\\  0&;\;\mathrm{Otherwise}\end{cases}$, and $c_2(z) = \begin{cases} [1,-1]^\tran&;\;3\pi/4\leq z\leq\pi\\ 0&;\;\mathrm{Otherwise}\end{cases}$.

Let us assume $b_i(z) =\sum_{k=5}^\infty\zeta_k^i\phi_k^i$, where $\zeta_k^1 = [\frac{1}{k}\;, \frac{1}{k}]^\tran$, and  $\zeta_k^2 = [\frac{1}{k^2}\;, -\frac{1}{k^2}]^\tran$ for $k>5$. Moreover, let $L_i(z)=b_i(z)$ $i=1,2$ (for all $z\in[0, \pi]$) represent actuator faults. Finally,  let  $\op{C} = [<c_1,\cdot>,\; <c_2,\cdot>]^\tran$, with $c_1$ and $c_2$ given above. As observed below the condition for the Case 1 stated in Section \ref{Sec:FDISpecialCase} does indeed hold.

In the following,  a detection filter is designed for detecting and isolating the fault $f_1$.
Since $\ssp{L}_2 = \spanset{L_2}\in\dom{A}$ and $\ssp{L}_2\not\in\ker \op{C}$, we obtain $\ssp{Z}^*=\ssp{Z}_1=\ssp{L}_2$ from the algorithm \eqref{Eq:T-CondSpaceAlg}.
Hence, one can write $\ssp{W}_\ell = 0$ (since $\ssp{L}_\ssp{N}=0$). Therefore, $\ssp{W}^*=\ssp{L}_2$. By setting $\ssp{W}_{\phi,\mr{f}}^* = \ssp{W}_\mr{f}^*$ and since $c_1\perp\phi_k^2$ for all $k\in\fld{N}$, $0\in\roi{A}$, we have $\ssp{N}+\ssp{L}_2 =\overline{\spanset{\phi_k^2}_{k\in\fld{N}}}$ (i.e., the unobservable subspace of the system \eqref{Eq:PDE} with only one input $y=c_1x$). Given that $\ssp{W}^*=\ssp{L}_2$, one obtains  $\op{S}_1^* = \overline{\spanset{\phi_k^2}_{k\in\fld{N}}}$. It follows that $\ssp{L}_1\cap\op{S}_1^*=0$, and a solution to the corresponding maps $\op{D}_1$ and $H_1$ is given by $\op{D}_1 = 0$ and $H_1 = [1,0]$. The factored out subsystem can therefore be specified by using the canonical projection on $\op{S}_1^*$, that is $\op{P}_1:\op{X}\rightarrow\op{X}/\op{S}_1^*$, as follows
\begin{equation}
\begin{split}
\dot{\omega}_1(t) &=\op{A}_p\omega(t)+\op{P}_1\op{B}u(t)+\op{P}_1\op{L}_1f_1(t),\\
y_\omega(t) &= \op{M}_1\omega_1(t),
\end{split}
\end{equation}
where $\omega_1\in\op{X}/\op{S}_1^*$, $u=[u_1,\;u_2]^\tran$, $y_\omega = H_1y$, $\op{A}_p$ and $\op{M}$ are solutions to the equations $\op{A}_p\op{P} = \op{P}\op{A}$ and $\op{MP}=H\op{C}$, respectively, and are given by
\begin{equation}
\op{A}_p = \frac{\partial^2 }{\partial z^2}+0.1,\;\;
\op{M}_1\omega_1= <c_2,\omega_1>.
\end{equation}

Since all the eigenvalues of $\op{A}_p$ are negative (the condition for Case 1 in the Subsection \ref{Sec:FDISpecialCase}), by using Theorem \ref{Thm:SpecialCase1FDICond}  a detection filter is therefore specified according to
\begin{equation}
\begin{split}
\dot{\omega}_1(t) &=\op{A}_o \omega_1(t)+ \op{P}_1\op{B}u(t),\\
r_1(t) &= H_1y(t) - \op{M}_1\omega_1(t),
\end{split}
\end{equation}
where $\op{A}_o = \op{A}_p$. In other words, the detection filter to detect and isolate the fault $f_1$ is given by
\begin{equation}
\frac{\partial \tilde{\omega}_1(t,z)}{\partial t} = \frac{\partial^2\tilde{\omega}_1(t,z)}{\partial z^2}+0.1\tilde{\omega}_1(t,z)+ b_{11}(z)\tilde{u}_1(t,z)+ b_{22}(z)\tilde{u}_2(t,z)
\end{equation}
where $\tilde{\omega}_1(t,z)\in\fld{R}$ is the corresponding function to $\omega_1(t)\in\op{X}$, $[b_{11}(z), b_{22}(z)]^\tran = \op{P}[b_1(z),b_2(z)]^\tran$. The error dynamics corresponding to the above detection filter (i.e., $e(t)=\op{P}_1x(t)-{\omega}_1(t)$) is given by $\dot{e}(t)=\op{A}_p e(t) + \op{P}_1\op{L}_1f_1(t)$. Therefore, if $f_1=0$, the error converges to zero exponentially. Otherwise, $e\neq 0$. The above residual (i.e, $r_1$) corresponding to the fault $f_1$ is also decoupled from $f_2$. By following along the same steps as above, one can also design a detection filter to detect and isolate the fault $f_2$. These details are not included for brevity.

For the purpose of simulations,  we consider a scenario where the fault $f_1$ with a severity of $2$ occurs at $t=5\;sec$ and the fault $f_2$ with a severity of $-1$ occurs at $t=7\;sec$. Figure \ref{Fig:States} depicts the states of the system \eqref{Eq:PDE} (namely, $\tilde{x}_1$ and $\tilde{x}_2$ with disturbances and noise signals $\nu_i$ and $w_i$ included in the simulations), and Figure \ref{Fig:Residuals} depicts the residuals $r_1$ and $r_2$. It clearly follows that $r_i$ is only sensitive to the fault $f_i$, $i=1,2$. Note that the thresholds are computed based on running $70$ Monte Carlo simulations for the \emph{healthy system}, where the thresholds are selected as the maximum residual signals $r_1$ and $r_2$ during the entire simulation runtime. The selected  thresholds are $th_1 = 0.09$ and $th_2 = 0.064$, corresponding to the residual signals $r_1$ and $r_2$, respectively. The faults $f_1$ and $f_2$ are detected at $t=5.051\;sec$ and $t=7.31\;sec$, respectively. Table \ref{Tab:FltDetectTime} shows the detection times corresponding to various fault severity  cases that are simulated. This table clearly shows the impact of the fault severity levels on the detection times. In other words, the lower the fault severity, the longer the detection time delay. Moreover, the minimum detectable fault severities associated with  $f_1$ and $f_2$ for this example are determined to be $0.05$ and $0.15$, respectively.
\begin{figure}
	\centering
	\begin{subfigure}{0.4\textwidth}
		\includegraphics[scale = 0.5]{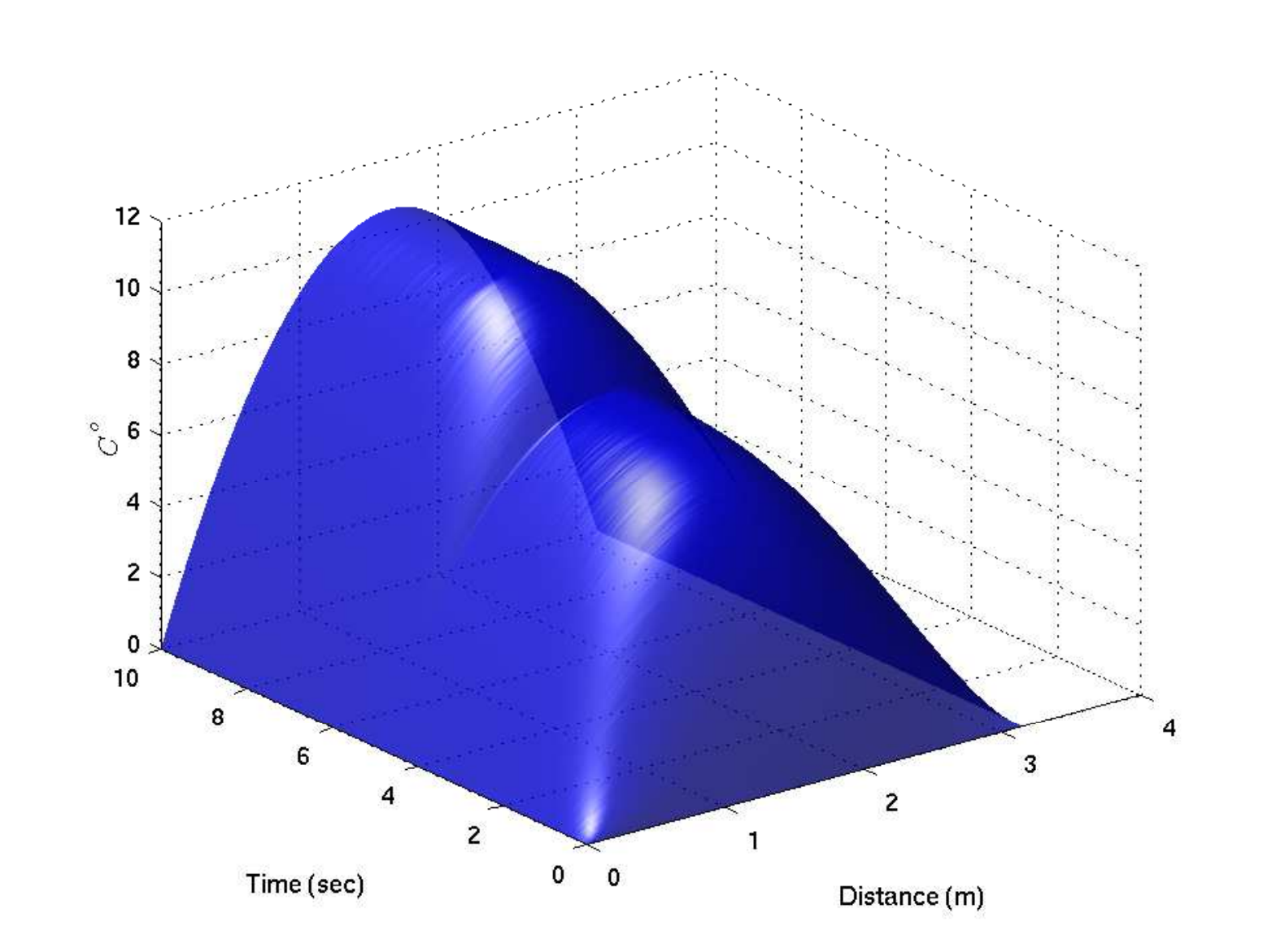}
		\caption{The state $\tilde{x}_1$.}
	\end{subfigure}
	~
	\begin{subfigure}{0.4\textwidth}
		\includegraphics[scale = 0.5]{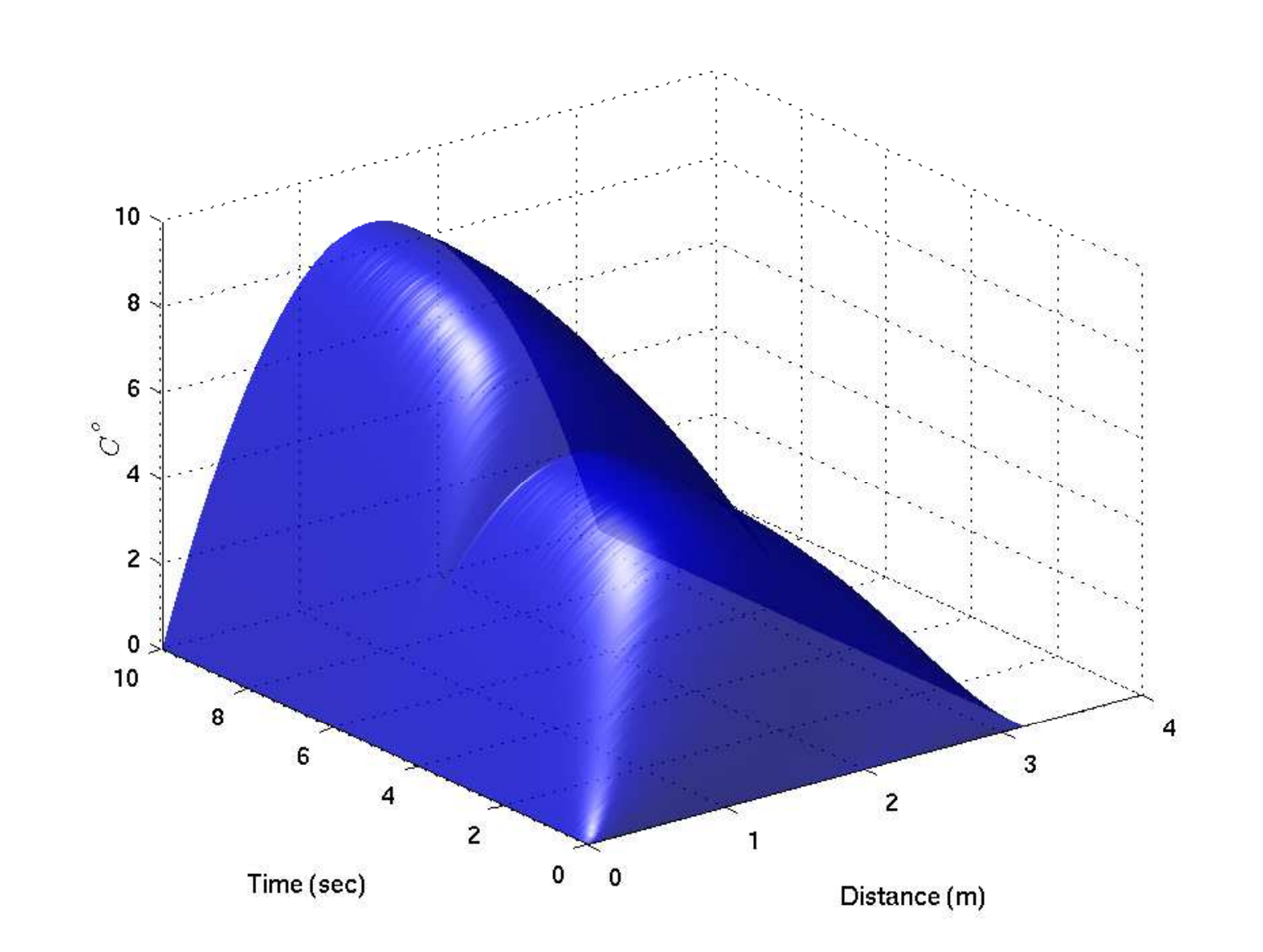}
		\caption{The state $\tilde{x}_2$.}
		
	\end{subfigure}
	\caption{The states of the system \eqref{Eq:PDE}. The faults $f_1$ and $f_2$ occur at $t=5\;sec$ and $t=7\;sec$ with severities of $2$ and $-1$, respectively.}
	\label{Fig:States}
	\vspace{-2mm}
\end{figure}
\begin{table}[h]
	\caption{Detection time delays of the faults $f_1$ and $f_2$ corresponding to various severities.}
	\centering
	\begin{tabular}{|l|c|c|}
		\hline
		\backslashbox{Severity}{Fault} &$f_1\;(sec)$ &$f_2\;(sec)$\\
		\hline
		$\begin{matrix}
		f_1 =2,\\
		f_2 =-1
		\end{matrix}$	&$0.051$	&$0.31$\\
		\hline
		$\begin{matrix}
		f_1 =0.5,\\
		f_2 =0.5
		\end{matrix}$	&$0.21$	&$0.555$\\
		\hline
		$\begin{matrix}
		f_1 =0.09,\\
		f_2 =0.2
		\end{matrix}$	&$1.18$	&$1.04$\\
		\hline
		$\begin{matrix}
		f_1 =0.05,\\
		f_2 =0.15
		\end{matrix}$	&$4.7$	&$1.34$\\
		\hline
	\end{tabular}
	\label{Tab:FltDetectTime}
\end{table}
\begin{figure}
	\centering
	\begin{subfigure}{0.4\textwidth}
		\includegraphics[scale = 0.5]{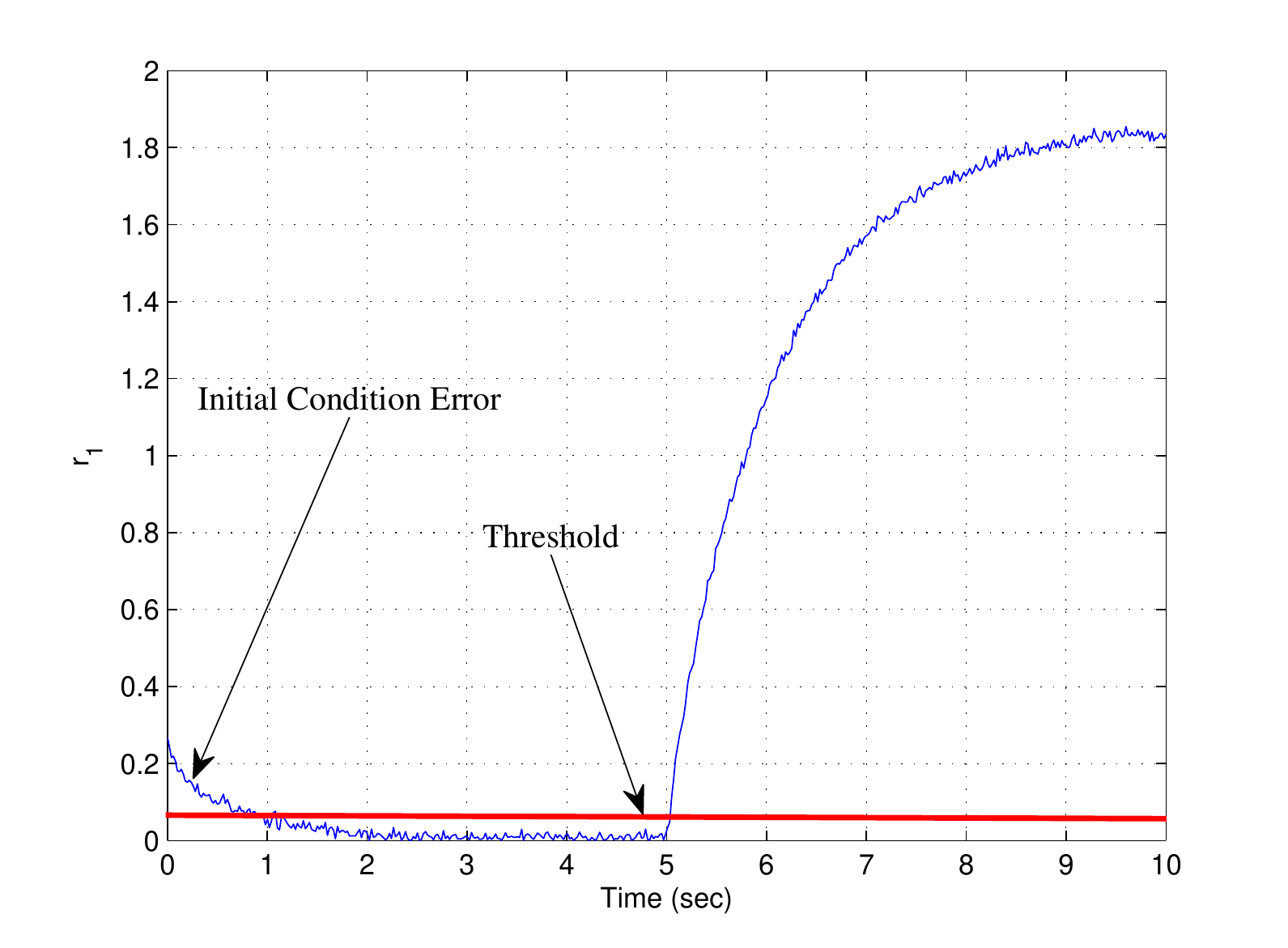}
		\caption{The residual signal $r_1$ for detecting and isolating the fault $f_1$.}
	\end{subfigure}
	~
	\begin{subfigure}{0.4\textwidth}
		\includegraphics[scale = 0.5]{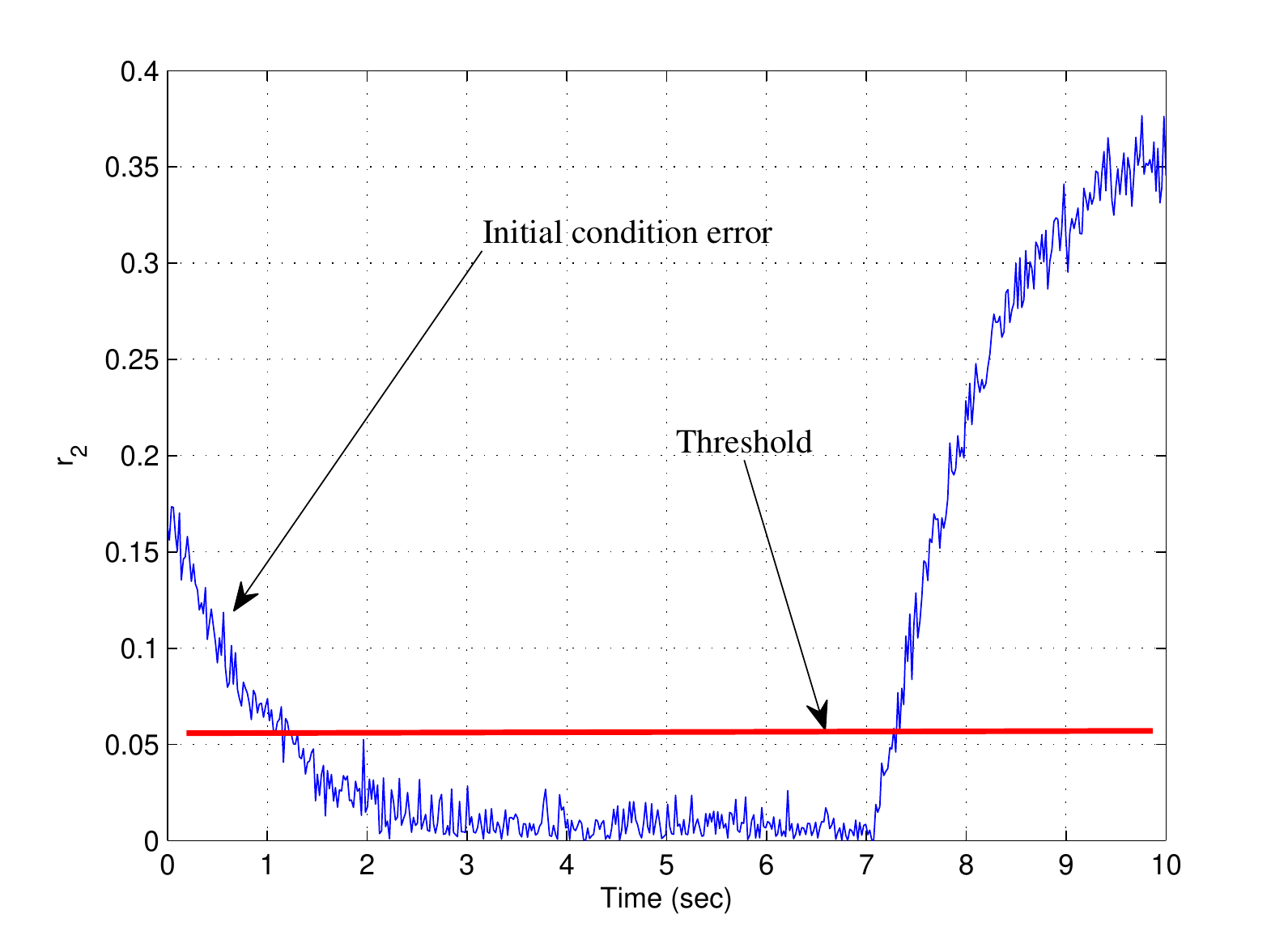}
		\caption{The residual signal $r_2$ for detecting and isolating the fault $f_2$.}
	\end{subfigure}
	\caption{The residual signals for detecting and isolating the faults $f_1$ and $f_2$. The faults occur at $t=5\;sec$ and $t=7\;sec$ with severities of $2$ and $-1$, respectively.}
	\label{Fig:Residuals}
	\vspace{-2mm}
\end{figure}

\begin{remark}
	When compared with {approximate} approaches that are developed in \cite{ACC2012,Armo_Demetrio} and \cite{Davis_Jor}  two main issues are worth pointing out:
	\begin{enumerate}
		\item The approximation of the system \eqref{Eq:FaultySys} is based on  only the operator $\op{A}$. As stated in \cite{Davis_Jor}, the system \eqref{Eq:FaultySys} was approximated by using the first two to four eigenvalues. However, since the fault signatures (namely, $L_1$ and $L_2$) in the above example have {no effect} on the eigenspaces of the first five eigenvalues, the faults $f_1$ and $f_2$ \underline{would not have been detectable} by using the approaches in \cite{ACC2012,Armo_Demetrio} and \cite{Davis_Jor}.
		\item In the references \cite{ACC2012,Armo_Demetrio} and \cite{Davis_Jor}, the Inf-D system is required to have eigenvalues  that are far in the left-half plane, that result in an extremely fast transient times (refer to Assumption 1 in \cite{ACC2012}), whereas our proposed approach in this paper {does not} suffer from this restriction and limitation.
	\end{enumerate}
\end{remark}

\section{Conclusion}\label{Sec:Conclusion}
In this paper, geometric characteristics associated with the regular Riesz spectral (RS) systems are investigated and new properties are introduced, specified, and developed. Specifically, various types of invariant subspaces such as the $\op{A}$- and $\fld{T}$-conditioned invariant and unobservability subspaces are developed and analyzed. Moreover, necessary and sufficient conditions for equivalence of  various conditioned invariant subspaces are also provided. Under certain conditions,  the algorithms corresponding to computing invariant subspaces are shown to indeed converge in a finite number of steps. Finally, we formulate and introduce the problem of fault detection and isolation (FDI) of RS systems, \emph{for the first time in the literature},  in  terms of  invariant subspaces.  For regular RS systems, we have developed and presented necessary and sufficient conditions for solvability of the FDI problem.
\bibliographystyle{IEEEtr}
\bibliography{InfiniteDim_Ref_Final}

\end{document}